\documentclass[letterpaper, 11pt,reqno]{amsart}
\usepackage{mathtools}
\usepackage{amsmath}
\usepackage{amssymb}
\usepackage{yhmath}
\usepackage{graphicx,overpic}
\usepackage{mathrsfs}
\usepackage{bbm}
\usepackage[dvipsnames]{xcolor}
\usepackage{tikz-cd}
\usepackage{tikz}
\usetikzlibrary{patterns}
\usepackage[
    hypertexnames=false,
    colorlinks=true,    
    linkcolor=NavyBlue,
    hypertexnames=false,
    filecolor=black,      
    urlcolor=blue,
    citecolor=OrangeRed,
    ]{hyperref}

\usepackage{verbatim}
\usepackage{float}

\DeclareMathAlphabet{\mathpzc}{OT1}{pzc}{m}{it}

\usepackage{thmtools}
\usepackage{thm-restate}

\usepackage{caption}

\usepackage[]{cleveref}
\newtheorem{theorem}{Theorem}[section]

\newtheorem*{claim*}{Claim}
\newtheorem{claim}{Claim}

\newtheorem{lemma}[theorem]{Lemma}

\newtheorem{corollary}[theorem]{Corollary}

\newtheorem{proposition}[theorem]{Proposition}

\theoremstyle{definition}

\theoremstyle{remark}

\newtheorem{Rmk}[theorem]{Remark}

\numberwithin{equation}{section}

\newcommand{\norm}[1]{\lVert#1\rVert}

\newcommand{\op}{\operatorname}

\newcommand{\be}{\begin{equation}}
\newcommand{\ee}{\end{equation}}
\newcommand{\Ga}{\Gamma}

\newcommand{\Z}{\mathbb Z}
\newcommand{\N}{\mathbb N}
\newcommand{\ga}{\gamma}

\newcommand{\la}{\lambda}

\newcommand{\inte}{\op{int}}

\newcommand{\cal}{\mathcal}
\newcommand{\br}{\mathbb R}

\renewcommand{\frak}{\mathfrak}

\newcommand{\e}{\varepsilon}

\renewcommand{\L}{\mathcal L}
\newcommand{\fa}{\mathfrak a}

\renewcommand{\i}{\op{i}}

\newcommand{\GL}{\op{SL}}

\newcommand{\fg}{\frak g}

\newcommand{\diag}{\op{diag}}

\renewcommand{\epsilon}{\e}

\newcommand{\SL}{\op{SL}}

\renewcommand{\L}{\mathcal L}

\renewcommand{\i}{\op{i}}
\newcommand{\cc}{\mathsf c}
\newcommand{\Dirac}{\mathsf D}

\title[Directional growth of coamenable normal subgroups]{Directional growth of coamenable normal subgroups: counterexamples and rigidity}
\author{Subhadip Dey}
\address{School of Mathematics, Tata Institute of Fundamental Research, Mumbai, India}
\email{subhadip@math.tifr.res.in}

\author{Hee Oh}
\address{Department of Mathematics, Yale University, New Haven, CT}
\email{hee.oh@yale.edu}
\thanks{
 Oh is partially supported by the NSF grant No. DMS-2450703.}
\author{Konstantinos Tsouvalas}
\address{Max Planck Institute for Mathematics in the Sciences in Leipzig,
    Inselstra{\ss}e 22,
    04103 Leipzig,
    Germany
}
\email{konstantinos.tsouvalas@mis.mpg.de}

\makeatletter
\let\@wraptoccontribs\wraptoccontribs
\makeatother

\begin{document}
\begin{abstract}
Roblin's theorem asserts that, in rank one, a coamenable normal subgroup
has the same critical exponent as its ambient group. Natural higher-rank
analogues would predict that coamenability preserves the limit cone and
the growth indicator. We show that both assertions fail, even when the
quotient is infinite cyclic. For every odd integer $n\geq 3$, we construct
a nonempty open family of Zariski-dense Borel--Anosov Schottky subgroups
of $\SL_n(\mathbb R)$ admitting cocyclic normal subgroups with strictly
smaller limit cones. Moreover, in $\SL_3(\mathbb R)$, we construct
cocyclic pairs with the same limit cone but distinct growth indicators
at an interior direction.

We then identify the precise rigidity that survives. Let $\Gamma$ be a
Zariski-dense Borel--Anosov subgroup of a connected semisimple real
algebraic group, and let $N\lhd\Gamma$ be coamenable. Then the growth
indicators of $N$ and $\Gamma$ agree on the fixed-point set of the
opposition involution, and their Riemannian critical exponents are equal.
The examples with equal limit cones show that the restriction to
opposition-invariant directions is sharp.
\end{abstract}

\maketitle


\section{Introduction}
A theorem of Roblin \cite{Roblin05} gives a robust form of growth rigidity in rank one:
if $\Gamma$ is a non-elementary discrete group of isometries of a negatively
curved symmetric space and $N\lhd\Gamma$ is coamenable, then their critical exponents coincide:
\[
    \delta_N=\delta_\Gamma.
\]
Here coamenability means that $\Gamma/N$ is amenable, or equivalently that
there is a $\Gamma$-invariant mean on $\ell^\infty(\Gamma/N)$. The purpose
of this paper is to determine what remains of this statement in higher
rank, where orbit growth is inherently directional.

Let $G$ be a connected semisimple real algebraic group, and let $(X,d)$ be
its Riemannian symmetric space. Fix a Cartan decomposition
\[
    G=K\exp(\fa^+)K,
\]
where $K$ is a maximal compact subgroup and $\fa^+$ is a positive Weyl
chamber. For $g\in G$, let $\mu(g)\in\fa^+$ be its Cartan projection. If
$o=[K]\in X=G/K$, then
\[
    d(o,go)=\norm{\mu(g)}.
\]
Thus $\mu(g)$ is a vector-valued displacement. The ordinary Riemannian
critical exponent records only its norm:
\begin{equation}\label{delta}
    \delta_\Gamma
    :=
    \limsup_{T\to\infty}
    \frac1T
    \log\#\{\gamma\in\Gamma:d(o,\gamma o)<T\}.
\end{equation}

Two basic invariants retain the directional information not contained in \eqref{delta}. The first is the limit cone
\begin{equation}\label{lc}
    \L_\Gamma
    =
    \left\{
    \lim_i t_i\mu(\gamma_i)\in\fa^+:
    t_i\to0,\ \gamma_i\in\Gamma
    \right\}.
\end{equation}
It records the asymptotic Cartan directions of elements of $\Gamma$; when
$\Gamma$ is Zariski dense, $\L_\Gamma$ is a convex cone with nonempty
interior \cite{Benoist97}. The second is the growth indicator
$\psi_\Gamma$, which records the exponential growth rate in each
direction in $\fa^+$. For $v\in\fa^+-\{0\}$,
\begin{equation}\label{gr}
\psi_\Gamma(v)
=
\norm{v}\inf_{\mathcal C\ni v}
\limsup_{T\to\infty}
\frac1T
\log\#\left\{\gamma\in\Gamma:
\norm{\mu(\gamma)}<T,\ \mu(\gamma)\in\mathcal C\right\},
\end{equation}
where the infimum is taken over open cones $\mathcal C\subset\fa^+$ containing
$v$, and $\psi_\Gamma(0)=0$. Its support is precisely $\L_\Gamma$.

A naive higher-rank extension of Roblin's theorem would assert that a
coamenable normal subgroup has the same limit cone and the same growth
indicator as its ambient group. These equalities were asserted in
Glorieux--Tapie \cite{GTa}. Our first results show that neither assertion is
true, even for cocyclic normal subgroups of Borel-Anosov Schottky groups.
The failure occurs in two logically distinct ways: the support of the
directional growth may shrink, and the growth rate may change even when the
support does not.

\subsection*{Robust failure of limit-cone rigidity}
We first construct an open family in which the limit cone of the normal
subgroup is strictly smaller. Let $F_2$ be the free group on two generators,
and let $\mathcal N$ be the normal closure of one of them.

\begin{restatable}{theorem}{LimitConeCounterexample}\label{m}
For every odd integer $n\ge3$, there exists a nonempty open subset
\[
    \Omega\subset\operatorname{Hom}(F_2,\SL_n(\mathbb R))
\]
such that, for every $\rho\in\Omega$,
\begin{enumerate}
\item $\rho(F_2)$ is a Zariski-dense Borel-Anosov subgroup of
$\SL_n(\mathbb R)$;
\item $\rho(F_2)/\rho(\mathcal N)\simeq\mathbb Z$;
\item $\L_{\rho(\mathcal N)}\subsetneq\L_{\rho(F_2)}$.
\end{enumerate}
\end{restatable}

For $G=\SL_n(\mathbb R)$, a finitely generated subgroup $\Gamma<G$ is
Borel-Anosov if there is $C>1$ such that, for every $\gamma\in\Gamma$ and
$1\le i\le n-1$,
\begin{equation}\label{borel}
    \alpha_i(\mu(\gamma))\ge C^{-1}|\gamma|-C,
\end{equation}
where $\alpha_i(v)=v_i-v_{i+1}$ and $|\gamma|$ is word length. In
particular, such a subgroup is discrete and its orbit maps are
quasi-isometric embeddings; see \cite{GuichardWienhard12,KLP17}.

Theorem \ref{m} shows that the expected directional rigidity already fails
for the simplest infinite amenable quotient, $\mathbb Z$, and that the
failure  is also stable under perturbation. Since the support of $\psi_\Gamma$ is
$\L_\Gamma$, it also gives
\[
    \psi_{\rho(\mathcal N)}\ne\psi_{\rho(F_2)}
    \qquad(\rho\in\Omega).
\]

\subsection*{Different growth with the same limit cone}
The preceding examples separate the growth indicators by separating their
supports. The next result shows that this is not the whole phenomenon: even
when the two groups have exactly the same asymptotic directions, their
rates of growth in those directions can differ.

\begin{restatable}{theorem}{GrowthIndicatorGap}\label{thm:growth-indicator-gap}
There exist a Zariski-dense Borel-Anosov subgroup
$\Gamma<\SL_3(\mathbb R)$ and a cocyclic Zariski-dense normal subgroup
$N\lhd\Gamma$ such that
\[
    \L_N=\L_\Gamma
\]
and
\[
    \psi_N(v)\ne\psi_\Gamma(v)
    \qquad\text{for some }v\in\inte\L_N.
\]
\end{restatable}

Thus the counterexamples are not merely caused by a loss of available
Cartan directions. They exhibit a genuine change in directional growth
inside a common limit cone.

\subsection*{The rigidity that survives}
We now describe the exact remnant of the rank-one theorem. Let
$\i:\fa^+\to\fa^+$ be the opposition involution, characterized by
\[
    \mu(g^{-1})=\i(\mu(g)).
\]
For a linear form $\varphi\in\fa^*$, define the associated directional
critical exponent by
\[
\delta_{\Gamma,\varphi}
:=
\limsup_{T\to\infty}
\frac1T
\log\#\{\gamma\in\Gamma:\varphi(\mu(\gamma))<T\},
\]
and set
\[
    \bar\varphi:=\frac12(\varphi+\varphi\circ\i).
\]
The following symmetrized inequality is the main coamenability theorem of
the paper.

\begin{restatable}{theorem}{SymmetrizedCriticalExponent}\label{coa}
Let $G$ be a connected semisimple real algebraic group, let
$\Gamma<G$ be a Zariski-dense Borel-Anosov subgroup, and let
$N\lhd\Gamma$ be coamenable. If $\varphi\in\fa^*$ is positive on
$\L_\Gamma-\{0\}$, then
\[
    \delta_{\Gamma,\bar\varphi}
    \le \delta_{N,\varphi}
    \le \delta_{\Gamma,\varphi}.
\]
In particular, if $\varphi=\varphi\circ\i$, then
$\delta_{N,\varphi}=\delta_{\Gamma,\varphi}$.
\end{restatable}

The appearance of $\bar\varphi$ is not an artifact of the proof. It
identifies the universal symmetry that remains under passage to a
coamenable normal subgroup. By convex duality, Theorem \ref{coa} implies
the following rigidity statement for growth indicators.

\begin{restatable}{theorem}{GrowthIndicatorRigidity}\label{thm:growth-indicator-symmetry}\label{fd}
Let $G$, $\Gamma$, and $N$ be as in Theorem \ref{coa}. Then
\[
    \psi_N(v)=\psi_\Gamma(v)
    \qquad\text{whenever }\i(v)=v.
\]
In particular, if the opposition involution is trivial, then
$\psi_N=\psi_\Gamma$ on $\fa^+$.
\end{restatable}

The opposition involution is trivial precisely when $G$ has no simple
factor of type $A_n$ $(n\ge2)$, $D_{2n+1}$ $(n\ge2)$, or $E_6$; see
\cite[1.5.1]{Tits_classification}. Theorem
\ref{thm:growth-indicator-gap} shows that, when the involution is
nontrivial, the fixed-point condition in Theorem
\ref{thm:growth-indicator-symmetry} cannot in general be enlarged, even
under the additional assumption $\L_N=\L_\Gamma$.

The ordinary Riemannian critical exponent is governed by an
opposition-fixed maximizing direction. Consequently, the scalar part of
Roblin's theorem survives unchanged.

\begin{restatable}{theorem}{RiemannianCriticalExponentRigidity}\label{m3}
Let $G$ be a connected semisimple real algebraic group, let
$\Gamma<G$ be a Zariski-dense Borel-Anosov subgroup, and let
$N\lhd\Gamma$ be coamenable. Then
\[
    \delta_N=\delta_\Gamma.
\]
\end{restatable}

Together, these results give a sharp hierarchy: coamenability may fail to
preserve the limit cone, but it preserves the growth indicator in every
opposition-fixed direction and preserves the Riemannian critical exponent.

\subsection*{The role of the Borel-Anosov hypothesis}
The Borel-Anosov assumption has different roles in the negative and
positive results. In Theorems \ref{m} and
\ref{thm:growth-indicator-gap}, it strengthens the counterexamples: the
failure occurs for groups with uniform regularity, hyperbolic dynamics, and
well-controlled asymptotic geometry. In Theorem \ref{coa}, the assumption
provides the geometric structure used by the proof. For every
$\varphi>0$ on $\L_\Gamma-\{0\}$, the function
\[
    d_\varphi(\gamma_1,\gamma_2)
    =\varphi\bigl(\mu(\gamma_1^{-1}\gamma_2)\bigr)
\]
is proper and coarsely comparable to a word metric, while Cartan
projections are coarsely additive along word geodesics. These facts yield
the shadow and convolution estimates needed for amenable averaging. The
passage from directional critical exponents to growth indicators also uses
the strict concavity available for Borel-Anosov groups.

Thus the hypothesis is structural for the present proof rather than a
mere convenience. At the same time, Borel-Anosov is not the maximal scope
of the method. As explained in Remark \ref{general}, the argument extends
to settings in which the same coarse triangle and regularity estimates are
available, including relatively Morse subgroups and cusped Hitchin groups,
as well as the corresponding $\theta$-versions. Whether Theorem \ref{m3}
holds for arbitrary Zariski-dense discrete subgroups, without an
Anosov--Morse hypothesis, remains open.

In real rank one, Borel-Anosov subgroups are precisely convex cocompact
subgroups \cite{GuichardWienhard12}. Other examples include Hitchin
representations into split real semisimple groups
\cite{Labourie06,FockGoncharov11}, Schottky subgroups
\cite[Lemma 7.2]{ELO_anosov}, and self-joinings of convex cocompact
representations into products of rank-one groups.

\subsection*{Ideas of the proofs}
The counterexample in Theorem \ref{m} begins with a particularly simple
mechanism. Let $F_2=\langle a,b\rangle$, let
$\chi:F_2\to\mathbb Z$ satisfy $\chi(a)=1$ and $\chi(b)=0$, and let
$j:F_2\to\SL_2(\mathbb R)$ be Schottky. For $n=2d+1$, we combine the
irreducible $2d$-dimensional representation of $\SL_2$ with a
one-dimensional character twist:
\[
    \rho_\kappa(\gamma)
    =
    \begin{pmatrix}
    e^{-\kappa\chi(\gamma)/(2d)}\tau_{2d}(j(\gamma))&0\\
    0&e^{\kappa\chi(\gamma)}
    \end{pmatrix}.
\]
For sufficiently small $\kappa>0$, the middle Jordan coordinate satisfies
\[
    \lambda_{d+1}(\rho_\kappa(\gamma))=\kappa\chi(\gamma).
\]
It therefore vanishes on $\ker\chi$ but not on the ambient group. This
places the limit cone of the kernel in the hyperplane
$\{v_{d+1}=0\}$ while leaving an ambient direction outside that
hyperplane.

The model representation is reducible, so this observation alone does not
give the open Zariski-dense family in Theorem \ref{m}. The required
robustness is the first technical point of the paper. Since $\ker\chi$ has
infinite index and is not Anosov, the usual continuity of limit cones for
Anosov representations cannot be applied directly. We instead prove a
uniform perturbation estimate for the Jordan projections of every word in
the free group. It preserves the separation of normalized Jordan
directions under small perturbations and allows us to pass to an open set
of Zariski-dense Borel-Anosov representations.

For Theorem \ref{coa}, we work with weighted Poincar\'e series for the coarse
distance $d_\varphi$. Coarse additivity gives shadow estimates and a
weighted convolution inequality over cosets of $N$. An invariant mean on
$\Gamma/N$ averages the logarithmic distortion of the weights and produces
a character $\chi_\varphi:\Gamma\to\mathbb R$. Repeating the construction
for $\varphi\circ\i$ changes the sign of this character. Combining the two
estimates cancels it and leaves exactly the symmetrized form
$\bar\varphi$. Theorems \ref{thm:growth-indicator-symmetry} and \ref{m3}
then follow from convex duality and the geometry of the maximal growth
direction.

Finally, to prove Theorem \ref{thm:growth-indicator-gap}, we first show in
$\SL_3(\mathbb R)$ that every nonsymmetric positive linear form can detect
a strict gap between the dual critical exponents of a cocyclic subgroup and
its ambient group. In the model representation, this reduces to a
perturbation of the Schottky length by a character, and a pressure argument
shows that the critical exponent increases strictly away from the symmetric
form. We then adjoin a sufficiently large Schottky generator so that the
normal subgroup acquires nonsymmetric interior directions while the
critical-exponent gap persists. A finite-index refinement arranges
$\L_N=\L_\Gamma$.

\subsection*{Organization}
Section 2 proves the uniform perturbation estimate, and Section 3 applies
it to Theorem \ref{m}. Section 4 proves Theorem \ref{coa}; Section 5 derives
its two rigidity consequences. Section 6 establishes sharpness and proves
Theorem \ref{thm:growth-indicator-gap}. The appendix, by the third named
author, proves $\delta_N\ge\frac12\delta_\Gamma$ for every infinite normal
subgroup; its cone-localized form also gives \eqref{normalg}.

\subsection*{Acknowledgements}
This material is based upon work supported by the National Science
Foundation under Grant No. DMS-2424139 while the authors were in residence
at the Simons Laufer Mathematical Sciences Institute in Berkeley,
California, during the Spring 2026 semester. We thank Fanny Kassel for
asking whether the examples in Theorem \ref{thm:growth-indicator-gap} can
be arranged to satisfy $\L_N=\L_\Gamma$.

\section{Uniform perturbative control of Jordan projections}
The reducible model used in the proof of Theorem \ref{m} already
separates an ambient Jordan direction from all directions arising in the
cocyclic kernel. To turn that model into a nonempty open family of
Zariski-dense representations, however, the separation must persist
uniformly over infinitely many elements of the kernel. Ordinary continuity
of limit cones for Anosov representations does not provide this: the kernel
has infinite index and is not itself Anosov. The purpose of this section is
to establish the uniform perturbation estimate that replaces that
continuity argument.

Fix $n\ge 2$. For a matrix $g\in \op{SL}_n(\br)$, we denote by $\lambda_1(g)\geq \cdots \geq \lambda_n(g)$
the logarithms of the moduli of the eigenvalues of $g$, arranged in
non-increasing order. The Jordan projection $\la(g)$ is defined as 
$$
    \lambda(g)=\diag(\lambda_1(g),\ldots,\lambda_n(g)).
$$ 

 We denote by $(e_1,\ldots,e_n)$ the canonical basis of $\br^n$. We equip $\br^n$ with the
standard Euclidean norm $\|\cdot\|$; the same notation will also be used for the induced
operator norm on $\op{SL}_n(\br)$. We shall equip the projective space $\mathbb{P}(\mathbb{R}^n)$ with the angle metric $d_{\mathbb{P}}$ defined as follows $$d_{\mathbb{P}}([u],[v])=\sqrt[]{1-\langle u,v\rangle^2}, \qquad  ||u||=||v||=1$$ where $\langle \cdot, \cdot\rangle$ is the standard Euclidean inner product on $\mathbb{R}^n$. Given two compact subsets $S_1,S_2\subset \mathbb{P}(\mathbb{R}^n)$, $\textup{dist}(S_1,S_2)$ denotes the Hausdorff distance between $S_1$ and $S_2$.

\subsection*{Proximal elements} A matrix $g\in \op{SL}_n(\br)$ is called \emph{proximal} if  $  \lambda_1(g)>\lambda_2(g)$. 
In this case, $g$ has a unique attracting point
$x_g^+\in \mathbb P(\br^n)$ and a repelling hyperplane $V_g^-$ such that, for any
$x\in \mathbb P(\br^n)- \mathbb P(V_g^-)$,
$$
    \lim_{k\to\infty} g^k x=x_g^+.
$$
A matrix $g$ is called \emph{biproximal} if both $g$ and $g^{-1}$ are proximal. In this
case we use the notation
$$
    x_g^-\coloneqq x_{g^{-1}}^+,
    \qquad
    V_g^+\coloneqq V_{g^{-1}}^-.
$$

Denote by $\mathcal{F}_{1,n-1}(\mathbb{R})$ the space of $(1,n-1)$-flags in $\mathbb{R}^n$. Two such flags $(x_1,V_1)$ and $(x_2,V_2)$ are {\em antipodal} if $x_1\notin V_2$ and $x_2\notin V_1$. 
For a biproximal element $g\in \SL_n(\mathbb R)$, the pair
$ (x_g^+,V_g^+) \in \mathcal{F}_{1,n-1}(\mathbb{R})$
is the attracting fixed flag of $g$, where $x_g^+$ is the attracting
fixed point of $g$ on $\mathbb P(\mathbb R^n)$, and $V_g^+$ is the
attracting fixed hyperplane of $g$, equivalently the attracting fixed point
of $\wedge^{n-1}g$ under the Pl\"ucker embedding. Similarly,
$(x_g^-,V_g^-)\in \mathcal{F}_{1,n-1}(\mathbb{R})$ is the repelling fixed flag.

The following elementary fact records the stability of attracting and repelling data under
small perturbations. 

\begin{lemma}\label{fact}
Let $h_0\in \op{SL}_n(\br)$, $n\geq 2$, be a biproximal matrix of the form
$$
    h_0=
    g
    \begin{pmatrix}
        \mu_1 & &\\
        & A &\\
        & & \mu_n
    \end{pmatrix}
    g^{-1},
    \qquad
    g\in \op{SL}_n(\br),\quad A\in \op{GL}_{n-2}(\br).
$$
with attracting and repelling fixed flags $$  \big([ge_1],\langle ge_1,\ldots,ge_{n-1}\rangle\big)   \quad\text{and}\quad    \big([ge_n],\langle ge_2,\ldots,ge_n\rangle\big),
$$
respectively. Given any $\epsilon>0$, there exists a neighborhood of $h_0$ in
$\op{SL}_n(\br)$ contained in the set
\begin{multline*}
\Omega(h_0,\epsilon)\\
\coloneqq
\left\{
gh
{\setlength{\arraycolsep}{2pt}
\begin{pmatrix}
\mu_1' & &\\
& A' &\\
& & \mu_n'
\end{pmatrix}
(gh)^{-1}:
\begin{array}{ll}
|\mu_1-\mu_1'|<\epsilon,&
\|A^{\pm1}-(A')^{\pm1}\|<\epsilon,\\
\|h-I_n\|<\epsilon,&
|\mu_n^{-1}-(\mu_n')^{-1}|<\epsilon
\end{array}}
\right\}.
\end{multline*}
\end{lemma}

\begin{proof} This follows from the fact that, for any sequence $(h_n)_{n\in \mathbb{N}}$ in $\mathrm{GL}_n(\mathbb{R})$ with $\lim_n h_n=h_0$, then $h_n$ is also biproximal for large $n$, $\lim_n x_{h_n}^{\pm}=x_{h_0}^{\pm}$, $\lim_n V_{h_n}^{\pm}=V_{h_0}^{\pm}$ and $\lim_n V_{h_n}^{+}\cap V_{h_n}^{-}=V_{h_0}^{+}\cap V_{h_0}^{-}$.\end{proof}

\subsection*{Projective Anosov condition} A word-hyperbolic group $\Gamma<\textup{SL}_n(\mathbb{R})$ is called {\em projective Anosov} if there exist $c_0, c_1>0$ such that for all $\gamma\in \Gamma$, the Cartan projection $\mu(\ga)$ satisfies
$$\alpha_1(\mu(\gamma)) \geq c_0|\gamma|-c_1$$
where $\alpha_1(\op{diag}(v_1, \cdots, v_n))=v_1-v_2$
and $|\cdot |$ is the word length in $\Ga$ \cite{Labourie06,GuichardWienhard12,KLP17}.
By \cite{kassel-potrie}, this is equivalent to the existence of a constant $c>0$ such that, for all $\gamma\in \Gamma$, $$\lambda_1(\gamma)-\lambda_{2}(\gamma)\geq c|\gamma|_{\infty},$$ where $|\gamma|_{\infty}=\lim_{n\to \infty} \frac{|\gamma^n|}{n}$ is the stable translation length of $\gamma\in \Gamma$. 

\subsection*{Perturbative control of Jordan projections}
The next proposition provides the perturbative estimate needed to construct
the open family in Theorem \ref{m}. It says that, after replacing two
biproximal generators by sufficiently large powers, the first Jordan
coordinate of every word changes only by a prescribed multiplicative error
under small perturbations of the generators.
Although limit cones of Borel-Anosov subgroups are known to vary
continuously \cite{Sambarino14,kassel,DeyOh25}, this continuity cannot be
applied directly to the infinite-index normal subgroup, which is not Anosov.
The proposition below gives the required substitute: it controls the relevant
Jordan directions for all words in the normal subgroup under small
perturbations.

The proof is a uniform Schottky ping-pong estimate. It combines
control of attracting and repelling data with multiplicative spectral
bounds, following ideas of Abels--Margulis--Soifer
\cite{AbelsMargulisSoifer} and Benoist \cite{Benoist97,Benoist96}.

\begin{proposition}[Perturbative control of the top Jordan coordinate] \label{estimate}
Let $F_2$ be a free group generated by $a$ and $b$. Let ${\mathsf a},{\mathsf b}\in \op{SL}_n(\br)$, $n\geq 2$, be two biproximal matrices such that the four flags $(x_{\mathsf a}^\pm,V_{\mathsf a}^\pm)$, $(x_{\mathsf b}^\pm,V_{\mathsf b}^\pm)$ in $\mathcal F_{1,n-1}(\br)$ are pairwise antipodal. Given $\epsilon>0$, there exists
$m_0>0$ with the following property: for any $m\geq m_0$, there exists an open
neighborhood $\Omega_m$ of the representation
$$
    F_2\to \op{SL}_n(\br),
    \qquad
    a\mapsto {\mathsf a}^m,\quad b\mapsto {\mathsf b}^m,
$$
such that:
\begin{enumerate}
    \item\label{prop-item1} every representation in $\Omega_m$ is projective Anosov; 
    \item\label{prop-item2} for any 
    $\rho_1,\rho_2$ in $\Omega_m$ and any $h\in F_2$,
    $$
        (1-\epsilon)\lambda_1(\rho_1(h))
        \leq
        \lambda_1(\rho_2(h))
        \leq
        (1+\epsilon)\lambda_1(\rho_1(h)).
    $$
\end{enumerate}
\end{proposition}

\begin{proof}
The existence of $m_0>1$ such that, for all $m\geq m_0$, the subgroup
$\langle {\mathsf a}^m,{\mathsf b}^m\rangle$ is free and projective Anosov follows from \cite[Theorem 4.13]{KLP25}.
We now check that, after increasing $m_0$ and shrinking the neighborhood $\Omega_m$, if necessary, the estimate in (\ref{prop-item2}) also holds. 

 By the hypotheses on ${\mathsf a}$ and ${\mathsf b}$, after conjugating in
$\SL_n(\mathbb R)$, we may assume that
$$
    {\mathsf a}=
    \begin{pmatrix}
        \nu_1 & &\\
        & A &\\
        & & \nu_n
    \end{pmatrix},
    \qquad
    {\mathsf b}=
    g
    \begin{pmatrix}
        \mu_1 & &\\
        & B &\\
        & & \mu_n
    \end{pmatrix}
    g^{-1},
$$
where $A,B\in \op{GL}_{n-2}(\br)$ and $g\in \op{SL}_n(\br)$. Here $\nu_1$ and $\nu_n$ are the eigenvalues of
${\mathsf a}$ of maximum and minimum moduli, respectively, and $\mu_1$ and $\mu_n$ are the
corresponding eigenvalues of ${\mathsf b}$. With this normalization, 
$$
    (x_{\mathsf a}^+,V_{\mathsf a}^+)
    =
    \big([e_1],\langle e_1,\ldots,e_{n-1}\rangle\big),
    \qquad
    (x_{\mathsf a}^-,V_{\mathsf a}^-)
    =
    \big([e_n],\langle e_2,\ldots,e_n\rangle\big),
$$
$$
    (x_{\mathsf b}^+,V_{\mathsf b}^+)
    =
    \big([ge_1],\langle ge_1,\ldots,ge_{n-1}\rangle\big),
    \qquad
    (x_{\mathsf b}^-,V_{\mathsf b}^-)
    =
    \big([ge_n],\langle ge_2,\ldots,ge_n\rangle\big).
$$

The antipodality assumption implies that there exists $0<\theta<10^{-2}$ such that all
unit vectors in the directions of $g^{\pm1}e_1$ and $g^{\pm1}e_n$
have first and $n$-th coordinates of modulus at least $\theta$.

For $g\in \SL_n(\mathbb R)$, set
$$
    C_g:=2\|g\|\,\|g^{-1}\|\ge1.
$$
Then both $g$ and $g^{-1}$ act $C_g$-Lipschitzly on
$\mathbb P(\mathbb R^n)$: for unit vectors $v_1,v_2$,
$$
d_{\mathbb P}([gv_1],[gv_2])
\le
\left\|
\frac{gv_1}{\|gv_1\|}
-
\frac{gv_2}{\|gv_2\|}
\right\|
\le C_g\|v_1-v_2\|.
$$

Fix $m>1$ large and choose $\epsilon=\epsilon(m)>0$ satisfying
\begin{equation}\label{bound-epsilon}
0<\epsilon<
\min\left\{
\|A^{\pm m}\|,\,
\|B^{\pm m}\|,\,
(10^2C_g)^{-4},\,
\theta^{10}
\right\}.
\end{equation}

For $M\in \GL_{n-2}(\mathbb R)$, $\xi_1,\xi_n\in\mathbb R^\times$, and
$\epsilon>0$, set
$$
\mathcal B(M,\xi_1,\xi_n;\epsilon)
:=
\left\{
h
\begin{pmatrix}
t_1 & &\\
& M' &\\
& & t_n
\end{pmatrix}
h^{-1}
\;\middle|\;
\substack{
|t_1-\xi_1|<\epsilon,\ 
\|(M')^{\pm1}-M^{\pm1}\|<\epsilon\\
\|h-I_n\|<\epsilon,\ 
|t_n^{-1}-\xi_n^{-1}|<\epsilon
}
\right\}.
$$

Then define
\begin{align*}
    \Omega({\mathsf a}^m,\epsilon)
    &\coloneqq
    \mathcal B(A^m,\nu_1^m,\nu_n^m;\epsilon);\\
    \Omega({\mathsf b}^m,\epsilon)
    &\coloneqq
    g\,\mathcal B(B^m,\mu_1^m,\mu_n^m;\epsilon)\,g^{-1}.
\end{align*}
For the rest of the proof, if $h\in \SL_n(\mathbb R)$, we denote by
$$
    \ell_1(h)\ge \ell_2(h)
$$
the two largest moduli of the eigenvalues of $h$. We now record several
estimates.

\medskip
\noindent\textup{(i)}
Any two matrices
$ C\in \Omega({\mathsf a}^m,\epsilon)$ and $
    D\in \Omega({\mathsf b}^m,\epsilon)$
are biproximal. Moreover, their two largest eigenvalue moduli satisfy 
\begin{equation}\label{obs1-ineq1}
\min\Bigg\{
\frac{\ell_1(C^{\pm1})}{\ell_2(C^{\pm1})},
\frac{\ell_1(D^{\pm1})}{\ell_2(D^{\pm1})}
\Bigg\}
\geq E(m),
\end{equation}
where
\begin{equation}\label{obs1-ineq2}
E(m)\coloneqq
\frac{1}{4}
\min\Bigg\{
\frac{|\nu_1|^m}{\|A^m\|},\;
\frac{1}{|\nu_n|^m\|A^{-m}\|},\;
\frac{|\mu_1|^m}{\|B^m\|},\;
\frac{1}{|\mu_n|^m\|B^{-m}\|}
\Bigg\}.
\end{equation}
In addition, by the choice of $\epsilon>0$ and the definition of the projective metric
$d_{\mathbb P}$, the attracting and repelling fixed points and
hyperplanes of $C$ and $D$ satisfy 

$$
\begin{aligned}
\operatorname{dist}(x_C^\pm,x_{\mathsf a}^\pm)
&\le
\max_{\|h-I_n\|\le\epsilon}\max_{i=1,n}
d_{\mathbb P}([he_i],[e_i])
\le 10\epsilon
\le \sqrt\epsilon,\\
\operatorname{dist}(x_D^\pm,x_{\mathsf b}^\pm)
&\le
\max_{\|h-I_n\|\le\epsilon}\max_{i=1,n}
d_{\mathbb P}([ghe_i],[ge_i])
\le 10C_g\epsilon
\le \sqrt\epsilon,\\
\operatorname{dist}(\mathbb P(V_C^\pm),\mathbb P(V_{\mathsf a}^\pm))
&\le
\max_{\|h-I_n\|\le\epsilon}\max_{i=1,n}
d_{\mathbb P}([h^{-t}e_i],[e_i])
\le 10\epsilon
\le \sqrt\epsilon,\\
\operatorname{dist}(\mathbb P(V_D^\pm),\mathbb P(V_{\mathsf b}^\pm))
&\le
\max_{\|h-I_n\|\le\epsilon}\max_{i=1,n}
d_{\mathbb P}([(gh)^{-t}e_i],[g^{-t}e_i])
\le 10C_g\epsilon
\le \sqrt\epsilon,
\end{aligned}
$$
where $(\cdot)^{-t}$ denotes the inverse transpose.

\medskip
\noindent\textup{(ii)}
Recall that every unit vector representing one of the four lines
$$
    [ge_1],\ [ge_n],\ [g^{-1}e_1],\ [g^{-1}e_n]
$$
has first and $n$-th coordinates of modulus at least $\theta$.
 Define 
$$
    \mathcal C_1
    \coloneqq
    B_{5\sqrt{\epsilon}}(x_{\mathsf a}^+)
    \cup
    B_{5\sqrt{\epsilon}}(x_{\mathsf a}^-),
    \qquad
    \mathcal C_2
    \coloneqq
    B_{5\sqrt{\epsilon}}(x_{\mathsf b}^+)
    \cup
    B_{5\sqrt{\epsilon}}(x_{\mathsf b}^-).
$$

Since $\lim_{m\to\infty} E(m)^{1/m}>1$, after increasing $m_0$ we may
assume that, for every $m\ge m_0$, the following hold: 
\begin{enumerate}
    \item
    $$
        \min\left\{
        \ell_1({\mathsf a}^{\pm1})^{m\epsilon},
        \ell_1({\mathsf b}^{\pm1})^{m\epsilon}
        \right\}
        \geq
        2^{40}\|g\|^{10}\cdot\|g^{-1}\|^{10}\theta^{-10};
    $$
    \item for any $C\in \Omega({\mathsf a}^m,\epsilon)$, any
    $D\in \Omega({\mathsf b}^m,\epsilon)$, and any $p\in \mathbb Z^\ast$, the ratios
    $$
        \frac{\ell_1(C^{\pm1})}{\ell_2(C^{\pm1})},
        \qquad
        \frac{\ell_1(D^{\pm1})}{\ell_2(D^{\pm1})}
    $$
    are sufficiently large so that
    $$
    C^p
    \left(
    \mathbb P(\br^n)
  -
    \mathcal N_{\theta/2}\big(\mathbb P(V_C^+\cup V_C^-)\big)
    \right)
    \subset B_{\sqrt{\varepsilon}}(x_{C}^{\pm}) \subset
    \mathcal C_1;
    $$
    $$
    D^p
    \left(
    \mathbb P(\br^n)
-
    \mathcal N_{\theta/2}\big(\mathbb P(V_D^+\cup V_D^-)\big)
    \right)
    \subset B_{\sqrt{\varepsilon}}(x_{D}^{\pm})\subset
    \mathcal C_2.
    $$
\end{enumerate}

For the rest of the proof, fix $m\geq m_0$ and $\epsilon>0$ satisfying
\eqref{bound-epsilon}. Then
$$
\textup{dist}(\mathcal C_1,\mathbb{P}(V_D^\pm))
\geq
\textup{dist}(x_{\mathsf a}^\pm,\mathbb{P}(V_{\mathsf b}^+\cup V_{\mathsf b}^-))-6\sqrt{\epsilon}
\geq {\theta}/{2},
$$
and similarly,
$$ \textup{dist}(\mathcal C_2,\mathbb{P}(V_C^\pm))\geq {\theta}/{2}. $$
Consequently, for any $p\in \mathbb Z^\ast$,
\begin{equation}\label{pp-incl}
    C^p\mathcal C_2\subset \mathcal C_1,
    \qquad
    D^p\mathcal C_1\subset \mathcal C_2.
\end{equation}

Fix $h\in \op{SL}_n(\br)$ with $\|h-I_n\|<\epsilon$. For any
$[u_2]\in \mathcal C_2$, the first and $n$-th coordinates of the unit vector
$\frac{h^{-1}u_2}{\|h^{-1}u_2\|}$
have moduli at least $\theta/10$. Similarly, for any $[u_1]\in \mathcal C_1$,
the first and $n$-th coordinates of the unit vector
$\frac{h^{-1}g^{-1}u_1}{\|h^{-1}g^{-1}u_1\|}
$
have moduli at least $\theta/10$. For any
$C\in \Omega({\mathsf a}^m,\epsilon)$, any $D\in \Omega({\mathsf b}^m,\epsilon)$ and any
$p\in \mathbb Z^\ast$, \begin{align*} \frac{1}{4}\ell_1(C^p) &\leq ||C^p||\leq 4\ell_1(C^p);  \\
\frac{\ell_1(D^p)}{4||g||\cdot||g^{-1}||} &\leq ||D^p||\leq 4||g||\cdot||g^{-1}||\ell_1(D^p).\end{align*}

Thus, for any $[u_i]\in \mathcal C_i$, $i=1,2$ and $p\in \mathbb{Z}^{\ast}$, we have
\begin{align}\begin{split}\label{pp-ineq1}
    \|C^pu_2\|
    &\geq
    \frac{\theta\,\ell_1(C^p)\|u_2\|}{10^2}
    \geq
    \frac{\theta}{10^3}\|C^p\|\cdot \|u_2\|;\\
    \|D^pu_1\|
    &\geq
    \frac{\theta\,\ell_1(D^p)\|u_1\|}{10^2\|g\|\cdot\|g^{-1}\|}
    \geq
    \frac{\theta\,\|D^p\|\cdot \|u_1\|}
    {(10^2\|g\|\cdot\|g^{-1}\|)^2}.
\end{split}\end{align}

\medskip

\noindent\textup{(iii)}
Set
$$
    c:=(10^2\|g\|\cdot\|g^{-1}\|)^{-2}\theta.
$$
By the previous choices, for any
$$
    C,C'\in \Omega({\mathsf a}^m,\epsilon),
    \qquad
    D,D'\in \Omega({\mathsf b}^m,\epsilon),
$$
and any $p\in \mathbb Z^\ast$, we have
\begin{equation}\label{obs-iii}
    \frac{\|C^p\|}{\|{C'}^p\|^{1-\epsilon}}\geq \frac{1}{c},
    \qquad
    \frac{\|D^p\|}{\|{D'}^p\|^{1-\epsilon}}\geq \frac{1}{c}.
\end{equation}

Indeed, let $\delta\in\{-1,1\}$ and $p\in \mathbb N$. By the choice of $\epsilon$,
$$
\frac{\ell_1(C^{\delta p})}{\ell_1({C'}^{\delta p})}
=
\frac{\ell_1(C^\delta)^p}{\ell_1({C'}^\delta)^p}
\geq
\frac{1}{4^p},
\qquad
\frac{\ell_1(D^{\delta p})}{\ell_1({D'}^{\delta p})}
=
\frac{\ell_1(D^\delta)^p}{\ell_1({D'}^\delta)^p}
\geq
\frac{1}{4^p}.
$$
Moreover, since $\frac{1}{2}\ell_1({\mathsf a}^{\delta})^m \leq \ell_1(C^{\delta})\leq 2\ell_1({\mathsf a}^{\delta})^m$, $\frac{1}{2}\ell_1({\mathsf b}^{\delta})^m \leq \ell_1(D^{\delta})\leq 2\ell_1({\mathsf b}^{\delta})^m$, we conclude that
$$
    \frac{\ell_1({\mathsf a}^\delta)^{mp}}{2^{p+2}}
    \leq
    \|C^{\delta p}\|
    \leq
    2^{p+2}\ell_1({\mathsf a}^\delta)^{mp},
$$
and
$$
    \frac{\ell_1({\mathsf b}^\delta)^{mp}}
    {2^{p+2}\|g\|\cdot\|g^{-1}\|}
    \leq
    \|D^{\delta p}\|
    \leq
    2^{p+2}\|g\|\cdot\|g^{-1}\|\ell_1({\mathsf b}^\delta)^{mp}.
$$
Using these bounds for the matrices $C,D$ (and also for $C',D'$), we obtain
$$
\frac{\|C^{\delta p}\|}{\|{C'}^{\delta p}\|^{1-\epsilon}}
\geq
\frac{\|{C'}^{\delta p}\|^\epsilon}{4^{2p}}
\geq
\frac{\ell_1({\mathsf a}^\delta)^{mp\epsilon}}{2^{10p}}
\geq
\frac{1}{c},
$$
and similarly
$$
\frac{\|D^{\delta p}\|}{\|{D'}^{\delta p}\|^{1-\epsilon}}
\geq
\frac{\|{D'}^{\delta p}\|^\epsilon}{4^{2p}||g||^2\cdot||g^{-1}||^2}
\geq
\frac{\ell_1({\mathsf b}^\delta)^{mp\epsilon}}
{2^{10p}\|g\|^3\cdot\|g^{-1}\|^3}
\geq
\frac{1}{c}.
$$

We now prove the desired comparison estimate. Let $F_2=\langle a\rangle\ast \langle b\rangle$
and let $ \rho_1,\rho_2:F_2\to \op{SL}_n(\br)$
be representations such that
$$
    \rho_i(a)\in \Omega({\mathsf a}^m,\epsilon),
    \qquad
    \rho_i(b)\in \Omega({\mathsf b}^m,\epsilon),
    \qquad
    i=1,2.
$$
We show that $\rho_1$ and $\rho_2$ satisfy the estimate in
(\ref{prop-item2}), assuming $0<2\varepsilon<1$. Since
$\ell_1(\cdot)$ is conjugacy invariant, it suffices to consider cyclically reduced words.
For this, let
$$
    h=\prod_{j=1}^{l} a^{p_j}b^{q_j},
    \qquad
    p_j,q_j\neq 0.
$$
The other cyclically reduced forms are handled in the same way, after interchanging the
roles of $a$ and $b$ if necessary.
Using the ping-pong inclusions \eqref{pp-incl} and the estimates
in \eqref{pp-ineq1}, we obtain
\begin{equation}\label{pp-ineq3}
    \|\rho_1(h)\|
    \geq
    \prod_{j=1}^{l}
    \big(c\|\rho_1(a)^{p_j}\|\big)
    \big(c\|\rho_1(b)^{q_j}\|\big).
\end{equation}

Therefore, by \eqref{pp-ineq3} and the submultiplicativity of the operator norm,
$$
\frac{\|\rho_1(h)\|}{\|\rho_2(h)\|^{1-\epsilon}}
\geq
\prod_{j=1}^{l}
\left(
c\frac{\|\rho_1(a)^{p_j}\|}
{\|\rho_2(a)^{p_j}\|^{1-\epsilon}}
\right)
\left(
c\frac{\|\rho_1(b)^{q_j}\|}
{\|\rho_2(b)^{q_j}\|^{1-\epsilon}}
\right).
$$
By \eqref{obs-iii}, each factor on the right-hand side is at least $1$.
Hence
$${\|\rho_1(h)\|}\ge {\|\rho_2(h)\|^{1-\epsilon}}.
$$
Applying this inequality to $h^r$, $r\geq 1$, and using the fact that
$$
    \lim_{r\to\infty}\|\rho_i(h^r)\|^{1/r}
    =
    \ell_1(\rho_i(h)),
$$
we get
$$
    \ell_1(\rho_1(h))
    \geq
    \ell_1(\rho_2(h))^{1-\epsilon}.
$$
Since $\rho_1$ and $\rho_2$ were arbitrary, switching their roles gives
$$
    \ell_1(\rho_1(h))
    \leq
    \ell_1(\rho_2(h))^{1/(1-\epsilon)}
    \leq
    \ell_1(\rho_2(h))^{1+2\epsilon},
$$
provided $\epsilon$ is sufficiently small. Taking logarithms, we obtain
$$
    (1-\epsilon)\lambda_1(\rho_2(h))
    \leq
    \lambda_1(\rho_1(h))
    \leq
    (1+2\epsilon)\lambda_1(\rho_2(h)).
$$
This proves the desired
estimate. Finally, since ${\mathsf a}^m$ and ${\mathsf b}^m$ are biproximal, Lemma \ref{fact}, together with
the openness of the projective Anosov property \cite[Theorem 5.13]{GuichardWienhard12}, gives an open neighborhood $\Omega_m\subset \textup{Hom}(F_2,\op{SL}_n(\br))$
of the representation  $ a\mapsto {\mathsf a}^m,
    b\mapsto {\mathsf b}^m,$
consisting entirely of projective Anosov representations and contained in $\Omega({\mathsf a}^m,\epsilon)\times \Omega({\mathsf b}^m,\epsilon)$.
This completes the proof.
\end{proof}

The preceding proposition controls the first Jordan coordinate. For the proof of
Theorem \ref{m}, we need a corresponding estimate for the full Jordan projection.
This follows by applying Proposition \ref{estimate} to the fundamental representations.

Recall that $g\in \SL_n(\mathbb R)$ is called \emph{loxodromic} if it is
conjugate to a diagonal matrix whose eigenvalues have pairwise distinct
moduli. For $1\le k\le n-1$, the $k$-th exterior power representation
$\tau_k=\wedge^k\mathbb R^n$ is the $k$-th fundamental representation of
$\SL_n(\mathbb R)$, and
$$
    \lambda_1(\tau_k(g))
    =
    \lambda_1(g)+\cdots+\lambda_k(g).
$$
If two loxodromic elements have antipodal fixed flags in the full flag
variety, then their images under each $\tau_k$ are biproximal with
antipodal attracting and repelling flags. Applying Proposition
\ref{estimate} simultaneously to all $\tau_k$, and using that the full
Jordan projection is determined by the partial sums
$\lambda_1+\cdots+\lambda_k$, we obtain the following:

\begin{proposition}
    [Perturbative control of the Jordan projection]\label{cor-1}
Let ${\mathsf a},{\mathsf b}\in \op{SL}_n(\br)$, $n\geq 2$, be two loxodromic elements with
antipodal fixed flags in the full flag space $\mathcal F(\br^n)$. Given $\epsilon>0$,
there exist $m>1$ and open neighborhoods $\Omega_1,\Omega_2\subset \op{SL}_n(\br)$ of ${\mathsf a}^m$ and ${\mathsf b}^m$, respectively, such that any two representations
$    \rho,\rho':F_2\to \op{SL}_n(\br)$
with  $ \rho(a),\rho'(a)\in \Omega_1$ and $\rho(b),\rho'(b)\in \Omega_2$,
are Borel-Anosov and satisfy
$$
    \|\lambda(\rho(h))-\lambda(\rho'(h))\|
    \leq
    \epsilon\|\lambda(\rho(h))\| \quad\text{for any $h\in F_2$.}
$$
\end{proposition}

\section{Robust failure of limit-cone rigidity}
We now prove Theorem \ref{m}. The construction has two stages. First we
build a reducible Borel-Anosov representation in odd dimension for which a
single Jordan coordinate records the quotient map onto $\mathbb Z$ and
therefore vanishes on its kernel. We then apply the uniform estimate of
Section 2 to preserve the resulting separation of directions under small
perturbations, and finally pass to a nonempty open set of Zariski-dense
representations.

Let $F_2=\langle a\rangle\ast \langle b\rangle$ be a free group 
and let $$ {\chi}:F_2\to \mathbb Z$$
be the unique homomorphism satisfying  ${\chi}(a)=1$ and
    ${\chi}(b)=0$.
Then  $${\mathcal N}\coloneqq\ker{\chi}$$
is the normal closure of $b$ in $F_2$.
Write $ n=2d+1$ for $d\geq 1$.

\LimitConeCounterexample*

\begin{proof}
Fix a convex cocompact representation $j:F_2\to \op{SL}_2(\br)$, equivalently a projective Anosov representation.
Since $j$ is projective Anosov, there exists $c>0$ such that
$$
    \lambda_1(j(\gamma))\ge c|\gamma|_\infty
    \qquad(\gamma\in F_2).
$$
Since $\chi$ is a homomorphism and $|\chi(\eta)|\le |\eta|$ for every $\eta\in F_2$, for each $k\in\N$ we have
$$
    k|\chi(\gamma)|
    =
    |\chi(\gamma^k)|
    \le
    |\gamma^k|.
$$
Hence
$$
    |\chi(\gamma)|\le |\gamma|_\infty .
$$
It follows that
\be\label{sup}
    \sup_{\gamma\in F_2-\{e\}}
    \frac{|\chi(\gamma)|}{\lambda_1(j(\gamma))}
    <\infty.
\ee

Let
$$
    \tau_{2d}:\op{SL}_2(\br)\to \op{SL}_{2d}(\br)
$$
be the unique irreducible representation of dimension $2d$, up to conjugation. Then
$$
    \tau_{2d}\circ j:F_2\to \op{SL}_{2d}(\br)
$$
is Borel-Anosov. Moreover, for any $\gamma\in F_2$,
$$
    \lambda_d(\tau_{2d}(j(\gamma)))=\lambda_1(j(\gamma)).
$$

For $\kappa>0$, define a representation $\rho_\kappa:F_2\to \op{SL}_{2d+1}(\br)$ 
by
\be\label{rk}
    \rho_\kappa(\gamma)
    =
    \begin{pmatrix}
    e^{-\frac{\kappa{\chi}(\gamma)}{2d}}
    \tau_{2d}(j(\gamma))
    &\\
    &
    e^{\kappa{\chi}(\gamma)}
    \end{pmatrix},
    \qquad
    \gamma\in F_2.
\ee 
By \eqref{sup}, we may choose $\kappa>0$ sufficiently small so that
\be\label{choice}
    \kappa\frac{2d+1}{2d}    \sup_{\gamma\in F_2\smallsetminus\{e\}}
    \frac{|{\chi}(\gamma)|}{\lambda_1(j(\gamma))}
    <1.
\ee 
By \cite{lahn2025},
this condition ensures that $\rho_\kappa$ is Borel-Anosov in
$\op{SL}_{2d+1}(\br)$. 

We claim that
\be\label{middle} \lambda_{d+1}(\rho_\kappa(\gamma))
    =
    \kappa{\chi}(\gamma)
\quad\text{ for any $\gamma\in F_2$.}\ee 
Indeed, let $\gamma\ne e$, and set $L:=\lambda_1(j(\gamma))>0$ and $q:=\chi(\gamma)$.
The Jordan coordinates of $\tau_{2d}(j(\gamma))$ are
$$
    (2d-1)L,\ (2d-3)L,\ldots, L,\ -L,\ldots, -(2d-3)L,\ -(2d-1)L.
$$
After multiplying this $2d$-dimensional block by
$e^{-\kappa q/(2d)}$, all these coordinates are shifted by
$-\kappa q/(2d)$. Hence the middle two coordinates of the twisted
$\tau_{2d}$-block are
$$
    L-\frac{\kappa q}{2d}
    \qquad\text{and}\qquad
    -L-\frac{\kappa q}{2d}.
$$
The remaining one-dimensional block has logarithmic eigenvalue modulus
$$
    \kappa q.
$$
By the choice of $\kappa$ as in \eqref{choice}, we have
$$
    \frac{2d+1}{2d}\kappa |q|<L.
$$
Equivalently,
$$
    -L-\frac{\kappa q}{2d}
    <
    \kappa q
    <
    L-\frac{\kappa q}{2d}.
$$
Thus the coordinate $\kappa q$ lies strictly between the $d$-th and
$(d+1)$-st coordinates of the twisted $2d$-dimensional block. Therefore,
after arranging all $2d+1$ logarithmic eigenvalue moduli in decreasing
order, the middle coordinate is precisely
$$
    \lambda_{d+1}(\rho_\kappa(\gamma))=\kappa\chi(\gamma).
$$
The case $\gamma=e$ is immediate.

Recall that any infinite order element of a Borel-Anosov group is loxodromic. Since $\rho_\kappa(a)$ and $\rho_\kappa(b)$ are loxodromic with antipodal fixed flags in
the full flag space $\mathcal F(\br^n)$, \Cref{cor-1} gives an integer $m>0$
and an open neighborhood $\Omega_m\subset \textup{Hom}(F_2,\op{SL}_{2d+1}(\br))$
of the representation $\rho_\kappa':F_2\to \op{SL}_{2d+1}(\br)$ defined by
$$
    \rho_\kappa'(a)=\rho_\kappa(a^m),
    \qquad
    \rho_\kappa'(b)=\rho_\kappa(b^m),
$$
such that any representation in $\Omega_m$ is Borel-Anosov and, for any
$\sigma\in \Omega_m$ and any $w\in F_2$,
$$
    \|\lambda(\sigma(w))-\lambda(\rho_\kappa'(w))\|
    \leq
    \frac{\kappa}{80d^2\,\lambda_1(j(a))}
    \|\lambda(\rho_\kappa'(w))\|.
$$
We shall use the elementary inequality
$$
\left\|
\frac{x}{\|x\|}-\frac{y}{\|y\|}
\right\|
\le
2\frac{\|x-y\|}{\|y\|}
$$
for nonzero vectors $x,y$. Applying this with
$x=\lambda(\sigma(w))$ and $y=\lambda(\rho_\kappa'(w))$, and using the
preceding estimate,  we obtain that for any non-trivial $w\in F_2$,
\begin{equation}\label{main1-ineq1}
\left\|
\frac{\lambda(\sigma(w))}{\|\lambda(\sigma(w))\|}
-
\frac{\lambda(\rho_\kappa'(w))}{\|\lambda(\rho_\kappa'(w))\|}
\right\|
\leq
\frac{\kappa}{40d^2\,\lambda_1(j(a))}.
\end{equation}

Fixing $\sigma\in \Omega_m$, we claim 
\be\label{mL}
    \mathcal L_{\sigma({\mathcal N})}   \subsetneq
    \mathcal L_{\sigma(F_2)}\ee

We will first show:
\begin{equation}\label{hN}
\inf_{h\in {\mathcal N}- \{e\}}
\left\|
\frac{\lambda(\sigma(h))}{\|\lambda(\sigma(h))\|}
-
\frac{\lambda(\sigma(a))}{\|\lambda(\sigma(a))\|}
\right\|
\ge
\frac{\kappa}{5d^2\,\lambda_1(j(a))}.
\end{equation}
 Let
 $ h=\prod_{j=1}^{l} a^{s_j}b^{t_j} \in {\mathcal N} $ be any non-trivial element.
Then $ s_1+\cdots+s_l=0$ and hence, by the definition of $\rho_\kappa'$,
$$
    \lambda_{d+1}(\rho_\kappa'(h))=\lambda_{d+1} ( \prod_{j=1}^{l}\rho_{\kappa}(a^{ms_j}b^{m t_j}))=\kappa \,{\chi}( \prod_{j=1}^{l} a^{ms_j}b^{mt_j})=\kappa m\sum_{i=1}^l s_i=0.
$$
Therefore \eqref{main1-ineq1} implies:
$$
    \sup_{h\in \mathcal N\smallsetminus\{e\}}
    \frac{|\lambda_{d+1}(\sigma(h))|}
    {\|\lambda(\sigma(h))\|}
    \leq
    \frac{\kappa}{40d^2\,\lambda_1(j(a))}.
$$

On the other hand, since $\chi(a)=1$, we obtain from \eqref{middle} and the choice of $\kappa$:
$$
    \lambda_{d+1}(\rho_\kappa(a))=\kappa \quad\text{and}\quad \frac{\kappa}{2d}\le \frac{\lambda_1(j(a))}{2d+1}.
$$
The largest absolute value of a Jordan coordinate of
$\tau_{2d}(j(a))$ is $(2d-1)\lambda_1(j(a))$. Since the
$2d$-dimensional block in $\rho_\kappa(a)$ is shifted by
$-\kappa/(2d)$, and the remaining one-dimensional block contributes
$\kappa$. Thus, for every $1\leq j\leq 2d+1$,
\[
\begin{aligned}
|\lambda_j(\rho_\kappa(a))|
&\leq |\lambda_1(\tau_{2d}(j(a)))|+\frac{\kappa}{2d}\\
&\leq \left((2d-1)+\frac{1}{2d+1}\right)\lambda_1(j(a))\\
&\leq 2d\lambda_1(j(a)).
\end{aligned}
\]
Consequently,
\[
\|\lambda(\rho_\kappa(a))\|
\leq 2d\sqrt{2d+1}\,\lambda_1(\rho_\kappa(a))
\leq 4d^2\lambda_1(j(a)).
\]
Since $\rho'_\kappa(a)=\rho_\kappa(a^m)$, we conclude that
$$
    \frac{\lambda_{d+1}(\rho_\kappa'(a))}
    {\|\lambda(\rho_\kappa'(a))\|}
    =
    \frac{m\lambda_{d+1}(\rho_\kappa(a))}
    {\|m\lambda(\rho_\kappa(a))\|}
    \geq
    \frac{\kappa}{4d^2\,\lambda_1(j(a))}.
$$
Again applying \eqref{main1-ineq1}, we get
$$
    \frac{\lambda_{d+1}(\sigma(a))}
    {\|\lambda(\sigma(a))\|}
    \ge
    \frac{\lambda_{d+1}(\rho_\kappa'(a))}
    {\|\lambda(\rho_\kappa'(a))\|}
    -
    \frac{\kappa}{40d^2\,\lambda_1(j(a))}
    \ge
    \frac{9\kappa}{40d^2\,\lambda_1(j(a))}.
$$
It follows that, for every $h\in\mathcal N-\{e\}$,
$$
\begin{aligned}
\left\|
\frac{\lambda(\sigma(a))}{\|\lambda(\sigma(a))\|}
-
\frac{\lambda(\sigma(h))}{\|\lambda(\sigma(h))\|}
\right\|
&\ge
\left|
\frac{\lambda_{d+1}(\sigma(a))}
{\|\lambda(\sigma(a))\|}
-
\frac{\lambda_{d+1}(\sigma(h))}
{\|\lambda(\sigma(h))\|}
\right| \\
&\ge
\frac{9\kappa}{40d^2\,\lambda_1(j(a))}
-
\frac{\kappa}{40d^2\,\lambda_1(j(a))} \\
&=
\frac{\kappa}{5d^2\,\lambda_1(j(a))}.
\end{aligned}
$$
This proves \eqref{hN}.
It follows that the direction of $\lambda(\sigma(a))$ is not
contained in the set of accumulation directions of the Jordan projections of
elements of $\sigma(\mathcal N)$. By Theorem \ref{thm:Benoist-Jordan-limit-cone} below,
 \eqref{hN} implies \eqref{mL}.

Since every representation in $\Omega_m$ is Borel-Anosov and $F_2$ has no torsion, it is faithful. Therefore
$$
\sigma(F_2)/\sigma(\mathcal N)\simeq F_2/\mathcal N\simeq \mathbb Z.
$$
Finally, by \cite[Theorem 8.2]{burger_Inv}, Zariski-dense representations
form an open dense subset of
\[
    \operatorname{Hom}(F_2,\SL_{2d+1}(\br)).
\]
After shrinking $\Omega_m$, we may therefore choose a nonempty open subset
$\Omega\subset\Omega_m$ consisting entirely of Zariski-dense
representations.
For every $\sigma\in\Omega$, the subgroup $\sigma(\mathcal N)$ is
normal and infinite in $\sigma(F_2)$. Its Zariski closure is therefore a
normal infinite algebraic subgroup of $\SL_{2d+1}(\br)$. Since
$\SL_{2d+1}(\br)$ is simple, this closure is all of
$\SL_{2d+1}(\br)$.
\end{proof}

\begin{Rmk}
The construction necessarily uses the fact that the opposition involution
is nontrivial for $\SL_n(\br)$; see
\Cref{lem:symmetric-directions-normality}.
\end{Rmk}

\section{Coamenability and symmetrized critical exponents}
Let $G$ be a connected semisimple real algebraic group.
Let $A$ be a maximal real split torus of $G$.
Let $\fg$ and $\fa$ respectively denote the Lie algebras of $G$
and $A$. 

\subsection*{Cartan and Jordan projections}
Fix a positive Weyl chamber $\fa^+ \subset \fa$ and set $A^+=\exp \fa^+$, and a maximal compact subgroup $K< G$ such that the Cartan decomposition $G=K A^+ K$ holds in the sense that
 for any $g\in G$, there exists a unique element $\mu(g)\in \fa^+$ such that $$g\in K \exp \mu(g) K.$$
The map $G\to \fa^+$ given by $g\mapsto \mu(g)$ is called the Cartan projection. 
Its basic property (\cite[Lemma 4.6]{Benoist97}) is that  for any compact subset $Q \subset G$, there exists $C=C(Q)>0$ such that for all $g \in G$, \be  \label{lem.cptcartan} \sup_{q_1, q_2\in Q} \| \mu(q_1gq_2) -\mu(g)\| \le C .\ee 

We also use the Jordan projection. If $g=g_e g_h g_u$ is the
multiplicative Jordan decomposition of $g$, with $g_e$ elliptic,
$g_h$ hyperbolic, and $g_u$ unipotent, then $g_h$ is conjugate to a
unique element $\exp\lambda(g)$ with $\lambda(g)\in\fa^+$. The vector
$\lambda(g)$ is called the Jordan projection of $g$. Equivalently,
\[
    \lambda(g)=\lim_{m\to\infty}\frac{1}{m}\mu(g^m).
\]

\subsection*{Opposition involution}
Let $N_K(A)$ denote the normalizer of $A$ in $K$. Fix an element $w_0\in N_K(A) $ of order $2$  representing the longest Weyl element so that $\op{Ad}_{w_0}\mathfrak a^+= -\mathfrak a^+$. 
The map \be\label{oppo} \i= -\op{Ad}_{w_0}:\fa\to \fa\ee  is called the opposition involution. 
  It preserves $\fa^+$.
 We have 
\be\label{inverse}  \mu(g^{-1})=\i (\mu(g))\quad \text{ and } \quad \la(g^{-1})=\i (\la(g))\quad\text{ for all $g\in G$. }
\ee
It follows that the limit cone $\L_\Ga$ of any closed subgroup $\Ga<G$, defined in \eqref{lc}, is preserved under $\i$.

We also have the following theorem of Benoist:
\begin{theorem} {\cite{Benoist97}}\label{thm:Benoist-Jordan-limit-cone}
Let $\Gamma<G$ be a Zariski dense discrete subgroup. Then
$\mathcal L_\Gamma$ is the smallest closed cone containing the Jordan
projections of  elements of $\Gamma$. Equivalently, it is the smallest
closed cone containing the Jordan projections of the loxodromic elements of $\Gamma$.
\end{theorem}

\subsection*{Symmetric directions from normality}
The opposition involution plays an important role in what follows. The
reason is that normality always forces the $\i$-symmetric part of the
ambient limit cone to appear in the normal subgroup. Thus the phenomenon in
\Cref{m} is possible only in directions which are not fixed by $\i$
In particular, this explains why the construction requires a group for
which the opposition involution is non-trivial.

For any cone $\cal C\subset \fa^+$,
let ${\cal C}^{\i}=\{x\in \cal C:\i(x)=x\}$ be its $\i$-fixed part.
\begin{lemma}
\label{lem:symmetric-directions-normality}
Let $G$ be a connected semisimple real algebraic group, and let
$\Gamma<G$ be a Zariski dense discrete subgroup. Let $N\lhd \Gamma$ be a
Zariski dense normal subgroup. Then
$$\L_N^{\i}=\L_\Ga^{\i}.$$
\end{lemma}

\begin{proof}
Let $\gamma\in\Gamma$ be loxodromic. Since $N$ is Zariski dense, we may
choose a loxodromic element $h\in N$ in general position with respect to
$\gamma$. In particular, the attracting and repelling flags of $h$ are
transverse to those of $\gamma$.
For $m\geq 1$, set
$$
    w_m=[\gamma^m,h]=\gamma^m h\gamma^{-m}h^{-1}.
$$
Then $w_m\in N$. By the standard Schottky product
estimate for loxodromic elements in general position \cite{Benoist97}, we have
$$
    \mu(w_m)
    =
    \mu(\gamma^m)+\mu(\gamma^{-m})+O(1),
$$
where the error is independent of $m$. Dividing by $m$, and using
$$
    \frac1m\mu(\gamma^m)\to \lambda(\gamma),
    \qquad
    \frac1m\mu(\gamma^{-m})\to \lambda(\gamma^{-1})
    =
    \i(\lambda(\gamma)),
$$
we obtain
$$
    \frac1m\mu(w_m)
    \to
    \lambda(\gamma)+\i(\lambda(\gamma)).
$$
Since $w_m\in N$, the limit belongs to $\mathcal L_N$.
By \Cref{thm:Benoist-Jordan-limit-cone}, the claim follows.
\end{proof}

\subsection*{Dual critical exponents} 
We say that a subgroup of $G$ is \emph{non-elementary} if it contains a
non-abelian free subgroup. By the Tits alternative, this is equivalent to
not being virtually solvable. In the word-hyperbolic setting, such a subgroup
has limit set in the Gromov boundary with at least three points; in
particular, it cannot fix a single boundary point.

Let $\Ga<G$ be a non-elementary discrete subgroup.
We denote by $\fa^*=\op{Hom}(\fa, \br)$, the space of linear forms on $\fa$. For $\varphi\in\mathfrak a^*$, define the critical exponent of $\Ga$ associated to the linear form $\varphi$ by
\be\label{dp}
\delta_{\Gamma,\varphi}
\coloneqq
\limsup_{T\to \infty}
\frac{1}{T}
\log
\#\{\gamma\in \Gamma : \varphi(\mu(\gamma))<T\}
\in [0,\infty].
\ee

Define the dual critical exponent function  $\delta_\Gamma^*:\mathfrak a^*\to [0,\infty]$
by $$\delta_\Gamma^*(\varphi)=\delta_{\Gamma,\varphi}.$$

\begin{corollary}\label{cm}
In the examples of Theorem \ref{m}, for every $\rho\in\Omega$,
$$
    \delta_{\rho(F_2)}^*\ne \delta_{\rho(\cal N)}^* .
$$
\end{corollary}
\begin{proof} By Theorem \ref{m}, we can
choose $v\in \mathcal L_{\rho(F_2)}- \mathcal L_{\rho(\cal N)}$.
Since $\rho(\cal N)$ is Zariski dense,
$\mathcal L_{\rho(\cal N)}$ is a closed convex cone with nonempty interior.
By the separating hyperplane theorem, there exists a linear form
$\psi \in \fa^*$ which is positive on $\L_{\rho(\cal N)}-\{0\}$ but satisfies $\psi(v)<0$. Then
$$
    \delta_{\rho(\cal N),\psi}<\infty
    \qquad\text{whereas}\qquad
    \delta_{\rho(F_2),\psi}=\infty. 
$$
\end{proof}

On the other hand, we will show that the dual critical exponents agree for
symmetric linear forms which are positive on $\L_\Ga-\{0\}$ (see Theorem \ref{coa}).

The growth indicator and the dual critical exponent function are related by
convex duality: for Zariski dense subgroups with common limit cone $\mathcal L$, equality
of the growth indicators on  the interior $\operatorname{int}\mathcal L$ is equivalent
to equality of the dual critical exponent functions on
$\operatorname{int}\mathcal L^{\vee}$. Here $\operatorname{int}\mathcal L^{\vee}$
is the interior of the dual cone, namely the set of linear forms on
$\mathfrak a$ which are positive on $\mathcal L-\{0\}$.
We will use the following variational formula.
\begin{lemma}\label{KMO_tent} \cite[Corollary 2.7]{KMO_tent}\label{tent}
If $\varphi\in \inte\L_\Ga^\vee$, then
 \be\label{var} 0<\delta_{\Ga, \varphi} =\sup_{v\in \L_\Ga-\{0\}}\frac{\psi_\Ga(v)}{\varphi(v)} <\infty\ee 
\end{lemma}

\subsection*{Coamenability and symmetrization}
For $\varphi\in\mathfrak a^*$, set
$$
    \bar\varphi
    :=
    \frac12(\varphi+\varphi\circ\op{i}).
$$
The main result of this section is the symmetrized inequality stated
in the introduction:
\SymmetrizedCriticalExponent*

We note that unless $\varphi$ is symmetric, the equality $\delta_{N,\varphi}=\delta_{\Gamma,\varphi}$ does not hold in general; see \Cref{ns0}.
\medskip 

We now prove Theorem \ref{coa}. Let $\Ga<G$ and $N\lhd\Gamma$ be as in
its statement, and fix
$\varphi\in\mathfrak a^*$ which is positive on $\mathcal L_\Gamma-\{0\}$.

We may assume without loss of generality that $\Ga$ is torsion-free.
Fix a word metric $d_w$ on $\Gamma$ and write $$|\ga|=d_w(e, \ga).$$
 Since $\Gamma$ is Borel-Anosov, it is word-hyperbolic. We denote by
$\partial\Gamma$ its Gromov boundary and by
$$
    \overline\Gamma:=\Gamma\cup\partial\Gamma
$$
its Gromov compactification.

\subsection*{The pseudo-distance $d_\varphi$}
For $\ga_1, \ga_2\in \Ga$, set
\be\label{tri}
d_\varphi(\gamma_1,\gamma_2)\coloneqq\varphi(
\mu(\gamma_1^{-1}\gamma_2)).
\ee 

We shall use the following coarse comparison with the word metric.
\begin{lemma}\label{qi}
There exist $A_\varphi\ge1$ and
$B_\varphi'\ge0$ such that, for all $\gamma_1,\gamma_2\in\Gamma$,
$$
    A_\varphi^{-1}d_w(\gamma_1,\gamma_2)-B_\varphi
    \le
    d_\varphi(\gamma_1,\gamma_2)
    \le
    A_\varphi d_w(\gamma_1,\gamma_2)+B_\varphi .
$$
\end{lemma}
\begin{proof}
Since $\Gamma$ is Borel-Anosov, the orbit map of $\Gamma$ into
the symmetric space is a quasi-isometric embedding. Equivalently, there
exist constants $A\ge1$ and $B\ge0$ such that
$$
    A^{-1}|\gamma|-B
    \le
    \|\mu(\gamma)\|
    \le
    A|\gamma|+B
    \qquad(\gamma\in\Gamma).
$$
On the other hand, since $\varphi$ is positive on
$\mathcal L_\Gamma-\{0\}$, there exists $c_\varphi>0$ such that
$$
    \varphi(v)\ge c_\varphi\|v\|
    \qquad(v\in\mathcal L_\Gamma).
$$
Since $\mu(\Gamma)$ stays within a
bounded distance of $\mathcal L_\Gamma$ \cite{Benoist97}, there exists
$B'_\varphi\ge0$ such that
$$
    \varphi(\mu(\gamma))
    \ge
    c_\varphi\|\mu(\gamma)\|-B'_\varphi
    \qquad(\gamma\in\Gamma).
$$
The reverse inequality $\varphi(\mu(\gamma))\le \|\varphi\|\,\|\mu(\gamma)\|$
is immediate. This proves the lemma.
\end{proof}

\begin{proposition} \cite[Corollary 5.15]{LeeOh23} \label{Grprod}
 Given $C\ge 0$, there exists 
$L>0$ such that, for any $\gamma_1,\gamma_2\in\Gamma$ satisfying
$    \big||\gamma_1\gamma_2|-|\gamma_1|-|\gamma_2|\big|\leq C$,
we have
$$
    \big\|
    \mu(\gamma_1\gamma_2)-\mu(\gamma_1)-\mu(\gamma_2)
    \big\|
    \leq L.
$$
\end{proposition}
Strictly speaking,  \cite[Corollary 5.15]{LeeOh23} was stated for $C=0$, but its proof works for a general $C>0$. Proposition
\ref{Grprod} implies that there
exists $D>0$ such that whenever $u$ lies on a word geodesic from $x$ to
$z$,
\be\label{t2}
    \left|d_\varphi(x,z)-d_\varphi(x,u)-d_\varphi(u,z)\right|\le D.\ee

Using the Gromov hyperbolicity, one can also deduce the following coarse triangle inequality for $d_\varphi$ from  Proposition
\ref{Grprod}. See \cite{DKO_AR} where a more general version was obtained:

\begin{proposition} \cite[Theorem 4.1]{DKO_AR} \label{t3}
For all $x,y,z\in\Gamma$,
\be\label{eq:coarse-triangle}
    d_\varphi(x,z)\le d_\varphi(x,y)+d_\varphi(y,z)+D.\ee 
\end{proposition}

\subsection*{Shadow estimate}
For $R>0$, let $\mathcal O_R(x,y)$ denote the word-shadow
$$
\mathcal O_R(x,y)
=
\{z\in\overline\Gamma:
\text{ some word geodesic }[x,z]\text{ meets }B_w(y,R)\}
$$
where $B_w(y,R)=\{x\in \Ga: d_w(x,y)<R\}$.

\begin{lemma}\label{lem:shadow-estimate}
Assume that $N\lhd\Gamma$ is non-elementary. Fix
$a>\delta_{N,\varphi}$. For a coset $C=gN\subset\Gamma$ and
$x\in\Gamma$, define a finite measure on $\overline\Gamma$, supported on
$C$, by
$$
    \mu_x^C
    \coloneqq
    \sum_{c\in C} e^{-a d_\varphi(x,c)}\Dirac_c,
$$
where $\Dirac_c$ denotes the Dirac measure at $c$.

Then there exist $R>0$ and $A\ge1$ such that, for every coset
$C=gN$ and all $x,y\in\Gamma$,
\be \label{eq:shadow-estimate}
A^{-1}|\mu_y^C|e^{-a d_\varphi(x,y)}
\le
\mu_x^C(\mathcal O_R(x,y))
\le
A|\mu_y^C|e^{-a d_\varphi(x,y)}.
\ee 
\end{lemma}

\begin{proof} For fixed $x,g\in\Gamma$, we have
$$
    \sup_{n\in N} \|\mu(x^{-1}ng)-\mu(n)\|<\infty.
$$
Since $a>\delta_{N,\varphi}$,
it follows that the measures $\mu_x^C$ are finite. 

We first prove the upper bound in \eqref{eq:shadow-estimate}. Let $c\in C\cap\mathcal O_R(x,y)$, and
choose $u\in[x,c]$ with $d_w(u,y)\le R$. By \eqref{t2}, applied to the
word geodesic $[x,c]$, we have
$$
    d_\varphi(x,c)\ge d_\varphi(x,u)+d_\varphi(u,c)-D.
$$
Since $d_w(u,y)\le R$, both $d_\varphi(u,y)$ and
$d_\varphi(y,u)$ are bounded by a constant depending only on $R$. Using
Proposition \ref{t3}, we get
$$
    d_\varphi(x,u)\ge d_\varphi(x,y)-O_R(1)
    \quad\text{and}\quad
    d_\varphi(u,c)\ge d_\varphi(y,c)-O_R(1).
$$
Therefore
$$
    d_\varphi(x,c)\ge d_\varphi(x,y)+d_\varphi(y,c)-O_R(1).
$$
It follows that
$$
    e^{-a d_\varphi(x,c)}
    \le
    A e^{-a d_\varphi(x,y)}e^{-a d_\varphi(y,c)}
$$
for some $A=A(R)>0$. Summing over $c\in C\cap\mathcal O_R(x,y)$ gives
the upper bound.

We prove the lower bound by contradiction. Suppose it fails. Then there are
$R_i\to\infty$,  cosets $C_i=g_i N$, and 
$x_i,y_i\in\Gamma$ such that
\be \label{eq:shadow-contra}
    e^{a d_\varphi(x_i,y_i)}
    \frac{\mu_{x_i}^{C_i}(\mathcal O_{R_i}(x_i,y_i))}
    {|\mu_{y_i}^{C_i}|}
    \to 0.\ee 
Left-translating by $x_i^{-1}$, and using normality of $N$, we may
assume that $x_i=e$. Write $h_i=y_i$. Since $d_\varphi$ is quasi-isometric
to the word metric, $h_i\to\infty$; otherwise, for large $i$, the shadow
$\mathcal O_{R_i}(e,h_i)$ would be all of $\overline\Gamma$, contradicting
\eqref{eq:shadow-contra}. Passing to a subsequence, we may assume that the sequence
$h_i^{-1}$ converges to some $\xi\in\partial\Gamma$.

Define probability measures on $\overline \Ga$:
$$ \nu_i \coloneqq \frac{1}{|\mu_{h_i}^{C_i}|}(h_i^{-1})_*\mu_{h_i}^{C_i}.
$$
After passing to a subsequence, $\nu_i\to\nu$ weakly. We claim that
$\nu=\delta_\xi$. Let $V\subset\overline\Gamma$ be a compact subset such that
$ V\cap\{\xi\}=\varnothing$. Since $h_i^{-1}\to\xi$, word
hyperbolicity gives
$$ V\subset \mathcal O_{R_i}(h_i^{-1},e) \quad\text{for all large $i$.} $$
 Equivalently,
$h_iV\subset \mathcal O_{R_i}(e,h_i)$. 
For $z\in h_iV$, Proposition \ref{t3}, applied to the triple
$(e,h_i,z)$, gives
$$
    d_\varphi(h_i,z)\ge d_\varphi(e,z)-d_\varphi(e,h_i)-D.
$$

Therefore
$$
\begin{aligned}
\nu_i(V)
&=
\frac{\mu_{h_i}^{C_i}(h_iV)}{|\mu_{h_i}^{C_i}|}  \le
e^{aD}e^{a d_\varphi(e,h_i)}
\frac{\mu_e^{C_i}(\mathcal O_{R_i}(e,h_i))}
{|\mu_{h_i}^{C_i}|}.
\end{aligned}
$$
By \eqref{eq:shadow-contra}, the right-hand side tends to $0$. Hence
$\nu_i(V)\to0$, and so $\nu(V)=0$. This implies that $\nu=\delta_\xi$.

We claim that for any $n\in N$, there exists
$C_n\ge1$, independent of $i$, such that 
for any non-negative
continuous function $f$ on $\overline \Ga$,
\be\label{cn}
    C_n^{-1}\nu_i(f)\le (n_*\nu_i)(f)\le C_n\nu_i(f).
\ee 
Indeed, the support of
$\nu_i$ is $h_i^{-1}C_i$, which is left $N$-invariant by normality.

Moreover, Proposition \ref{t3} gives a uniform comparison of the weights on
the support. For $z\in\Gamma$, we get
$$
    d_\varphi(e,n^{-1}z)
    \le d_\varphi(e,n^{-1})+d_\varphi(n^{-1},n^{-1}z)+D
    =
    d_\varphi(e,n^{-1})+d_\varphi(e,z)+D,
$$
and
$$
    d_\varphi(e,z)
    \le d_\varphi(e,n)+d_\varphi(e,n^{-1}z)+D.
$$
Hence
$$
    |d_\varphi(e,n^{-1}z)-d_\varphi(e,z)|
    \le
    \max\{d_\varphi(e,n),d_\varphi(e,n^{-1})\}+D.
$$
Thus, for a constant $C_n\ge1$, independent of $i$,
$$
    C_n^{-1}e^{-a d_\varphi(e,z)}
    \le
    e^{-a d_\varphi(e,n^{-1}z)}
    \le
    C_n e^{-a d_\varphi(e,z)}.
$$
Since $h_i^{-1}C_i$ is left $N$-invariant, changing variables
$w=nz$ in the sum defining $\nu_i$ gives, for any non-negative
continuous function $f$ on $\overline \Ga$,
$$
    C_n^{-1}\int f\,d\nu_i
    \le
    \int f(nz)\,d\nu_i(z)
    \le
    C_n\int f\,d\nu_i .
$$
This proves \eqref{cn}.
Passing to the weak limit gives 
$$ C_n^{-1}\nu(f) \le n_*\nu (f) \le C_n \nu (f)\quad \text{ for all $n\in N$.}$$
Since $\nu=\delta_\xi$, it follows that
any $n\in N$ fixes $\xi$. This contradicts
the non-elementarity of $N$. Hence this proves the lower bound.
\end{proof}

\subsection*{Weighted convolution estimate}
We next convert the shadow estimate into a weighted convolution inequality.
This is the estimate to which coamenability will be applied in the next
lemma.
\begin{lemma}\label{lem:convolution}\label{eq:conv-estimate}
With $a>\delta_{N,\varphi}$ as above, define, for any $x\in \Ga$,
$$
    U(x)\coloneqq|\mu_x^N|=\sum_{n\in N}e^{-a d_\varphi(x,n)}<\infty.
$$
Then for any $s>a$, there exists $b_s<\infty$ such that, for all
$x\in\Gamma$,
$$
    \sum_{y\in\Gamma}U(y)e^{-s d_\varphi(x,y)}
    \le b_sU(x).
$$
\end{lemma}

\begin{proof}
Let $R,A$ be the constants from Lemma \ref{lem:shadow-estimate}. By \Cref{qi}, we 
may choose a negative integer $m_0$ such that $d_\varphi(x,y)\ge m_0$ for all
$x,y\in\Gamma$, and set
$$
    S_m(x)\coloneqq\{y\in\Gamma:m\le d_\varphi(x,y)<m+1\},
    \qquad m\ge m_0.
$$

We first note that the family
$$
    \{\mathcal O_R(x,y):y\in S_m(x)\}
$$
has uniformly bounded multiplicity, independently of $x$ and $m$.
Indeed, suppose
$$
    z\in\mathcal O_R(x,y)\cap\mathcal O_R(x,y')
    \qquad\text{with }y,y'\in S_m(x).
$$
Choose points $u,u'$ on a word geodesic $[x,z]$ such that
$d_w(u,y)\le R$ and $d_w(u',y')\le R$. Since $d_\varphi$ is
quasi-isometric to $d_w$, all $d_\varphi$-distances between points at
$d_w$-distance at most $R$ are uniformly bounded. Hence, by Proposition
\ref{t3},
$$
    d_\varphi(x,u)=d_\varphi(x,y)+O_R(1)
    \quad\text{ and}\quad
    d_\varphi(x,u')=d_\varphi(x,y')+O_R(1).
$$
Since $y,y'\in S_m(x)$, it follows that
$$
    d_\varphi(x,u)=d_\varphi(x,u')+O_R(1).
$$
After interchanging $u$ and $u'$ if necessary, assume that $u$ lies
between $x$ and $u'$ on the geodesic $[x,z]$. Applying \eqref{t2} to
the geodesic segment $[x,u']$, we obtain
$$
    d_\varphi(u,u')
    \le d_\varphi(x,u')-d_\varphi(x,u)+D
    =
    O_R(1).
$$
By \Cref{qi}, this implies
$d_w(u,u')=O_R(1)$. Therefore $d_w(y,y')=O_R(1)$. Properness of the
word metric gives a uniform multiplicity bound; call it $M$.

Using the lower bound in \eqref{eq:shadow-estimate} for the coset $C=N$,
we get
$$
    U(y)e^{-a d_\varphi(x,y)}
    \le
    A\,\mu_x^N(\mathcal O_R(x,y)).
$$
Therefore, for any $m\ge m_0$,
$$
\sum_{y\in S_m(x)}
U(y)e^{-a d_\varphi(x,y)}
\le
A\sum_{y\in S_m(x)}\mu_x^N(\mathcal O_R(x,y))
\le
AMU(x).
$$
Since $d_\varphi(x,y)\ge m$ for $y\in S_m(x)$, we obtain
$$
\sum_{y\in S_m(x)}
U(y)e^{-s d_\varphi(x,y)}
\le
AMe^{-(s-a)m}U(x).
$$
Summing over $m\ge m_0$ proves the lemma.
\end{proof}

\subsection*{Amenable averaging over $\Ga/N$}
The next lemma is where coamenability enters the proof. 
Since $\Gamma/N$ is amenable, there exists a right $\Gamma$-invariant
mean
\be\label{mean} \mathfrak m:\ell^\infty(\Gamma/N)\to\mathbb R
\ee 
where $\Gamma$ acts on $\Gamma/N$ by right multiplication:
$(\gamma N)\cdot\gamma_0=\gamma\gamma_0N.
$
Here a mean on $\ell^\infty(\Gamma/N)$ means a positive normalized
linear functional. Thus $\mathfrak m(f)\ge0$ whenever $f\ge0$, and
$\mathfrak m(1)=1$. Right $\Gamma$-invariance means that
$$
    \mathfrak m(f\cdot\gamma_0)=\mathfrak m(f)
    \qquad
    (f\in\ell^\infty(\Gamma/N),\ \gamma_0\in\Gamma),
$$
where $(f\cdot\gamma_0)(\gamma N)=f(\gamma\gamma_0N)$.
The existence of such a mean is equivalent to the amenability of
$\Gamma/N$.

The idea, following Roblin \cite{Roblin05}, is to use
an invariant mean on $\Gamma/N$ to average the logarithmic distortion of
the weight $U$ under right translation; this produces a character
$\chi:\Gamma\to\mathbb R$. The two resulting twisted Poincar\'e series,
for $\varphi$ and for $\varphi\circ\op{i}$, can then be combined so that the
character cancels and the symmetrized exponent $\bar\varphi$ appears. 

The following lemma concludes the proof of Theorem \ref{coa}.
\begin{lemma}\label{lem:amenable-average}
 For any $s>\delta_{N,\varphi}$,
\be\label{eq:sym-series}
    \sum_{\gamma\in\Gamma}e^{-s d_{\bar\varphi}(e,\gamma)}<\infty.
\ee 

 In particular,  $\delta_{N,\varphi}\ge \delta_{\Gamma,\bar\varphi}.$
\end{lemma}

\begin{proof} Fix $s>\delta_{N,\varphi}$, and choose $a$ such that
$$
    \delta_{N,\varphi}<a<s.
$$
Let $U$ be the function defined in Lemma \ref{lem:convolution} for this
choice of $a$. 
First observe that $U$ descends to a function on $\Gamma/N$. Indeed, for
$n_0\in N$,
$$
\begin{aligned}
U(\gamma n_0)
&=
\sum_{n\in N}e^{-a d_\varphi(\gamma n_0,n)}
=
\sum_{n\in N}e^{-a d_\varphi(e,n_0^{-1}\gamma^{-1}n)} \\
&=
\sum_{n\in N}
e^{-a d_\varphi(e,\gamma^{-1}(\gamma n_0^{-1}\gamma^{-1})n)}
=
\sum_{n'\in N}e^{-a d_\varphi(e,\gamma^{-1}n')}
=
U(\gamma),
\end{aligned}
$$
where normality of $N$ was used in the third equality. Hence we may write
$U(\gamma N)\coloneqq U(\gamma)$.

For $\gamma_0\in\Gamma$, define a function $b_{\ga_0}$ on $\Ga/N$:
$$
    b_{\gamma_0}(\gamma N)
    \coloneqq
    \log U(\gamma\gamma_0)-\log U(\gamma),
$$
which is well-defined by the normality of $N$.  
Moreover,  $b_{\gamma_0}$ is bounded. To see this, let $\ga\in \Ga$ and $n\in N$.
By Proposition \ref{t3},
$$
    d_\varphi(\gamma\gamma_0,n)
    \ge
    d_\varphi(\gamma,n)-d_\varphi(e,\gamma_0)-D,
$$
and hence, summing over $n\in N$ gives
$$
    U(\gamma\gamma_0)
    \le
    e^{a(d_\varphi(e,\gamma_0)+D)}U(\gamma).
$$
The reverse inequality follows similarly, using $\gamma=\gamma\gamma_0\gamma_0^{-1}$.
Thus
$$
    |b_{\gamma_0}(\gamma N)|
    \le
    a\max\{d_\varphi(e,\gamma_0),d_\varphi(e,\gamma_0^{-1})\}+O(1),
$$
proving the boundedness of $b_{\ga_0}$. So $b_{\ga_0}\in \ell^\infty(\Gamma/N)$.

Using the mean  $\mathfrak m:\ell^\infty(\Gamma/N)\to\mathbb R$ given by \eqref{mean},
define $\chi:\Ga\to \br$ by
$$\chi(\gamma_0)\coloneqq\mathfrak m(b_{\gamma_0}) \quad\text{ for $\ga_0\in \Ga$}.
$$
The identity
$b_{\gamma_0\ga_1}(\gamma N)
    =
    b_{\ga_1}(\gamma\gamma_0N)+b_{\gamma_0}(\gamma N)
$
and the right $\Gamma$-invariance of $\mathfrak m$ imply
$$
\chi(\gamma_0\ga_1)=\chi(\ga_1)+\chi(\gamma_0) \quad\text{for all $\ga_0, \ga_1\in \Ga$.}
$$
Therefore $\chi:\Gamma\to\mathbb R$ is a character. Moreover, $\chi|_N=0$ since  $b_{\ga_0}=0$ for any $\ga_0\in N$.

By Jensen's inequality for the positive normalized mean $\mathfrak m$,
$$
e^{\chi(\gamma_0)}
=
e^{\mathfrak m(b_{\gamma_0})}
\le
\mathfrak m(e^{b_{\gamma_0}})
=
\mathfrak m\left(
\gamma N\mapsto
\frac{U(\gamma\gamma_0)}{U(\gamma)}
\right).
$$
Let $F\subset\Gamma$ be finite and let $s>a$. For every
$\gamma,\gamma_0\in\Gamma$, we have
$$
    d_\varphi(\gamma,\gamma\gamma_0)
    =
    \varphi(\mu(\gamma^{-1}\gamma\gamma_0))
    =
    \varphi(\mu(\gamma_0))
    =
    d_\varphi(e,\gamma_0).
$$
Hence, by Jensen's inequality, for each $\gamma_0\in F$,
$$
e^{\chi(\gamma_0)}e^{-s d_\varphi(e,\gamma_0)}
\le
\mathfrak m\left(
\gamma N\mapsto
\frac{U(\gamma\gamma_0)}{U(\gamma)}
e^{-s d_\varphi(\gamma,\gamma\gamma_0)}
\right).
$$
Summing over $\gamma_0\in F$, and using the positivity and linearity of
$\mathfrak m$, we obtain
$$
\begin{aligned}
\sum_{\gamma_0\in F}
e^{\chi(\gamma_0)}e^{-s d_\varphi(e,\gamma_0)}
&\le
\mathfrak m\left(
\gamma N\mapsto
\frac1{U(\gamma)}
\sum_{\gamma_0\in F}
U(\gamma\gamma_0)e^{-s d_\varphi(\gamma,\gamma\gamma_0)}
\right).
\end{aligned}
$$
By \Cref{eq:conv-estimate}, applied with
$x=\gamma$, we have
$$
\sum_{\gamma_0\in F}
U(\gamma\gamma_0)e^{-s d_\varphi(\gamma,\gamma\gamma_0)}
\le
\sum_{y\in\Gamma}
U(y)e^{-s d_\varphi(\gamma,y)}
\le
b_sU(\gamma).
$$
Therefore the function inside the mean is bounded above by $b_s$, and so
$$
\sum_{\gamma_0\in F}
e^{\chi(\gamma_0)}e^{-s d_\varphi(e,\gamma_0)}
\le b_s\quad\text{ and hence } \quad   \sum_{\gamma\in\Gamma}
    e^{\chi(\gamma)}e^{-s d_\varphi(e,\gamma)}<\infty.
$$

Changing variables $\gamma\mapsto\gamma^{-1}$  and using
$\chi(\gamma^{-1})=-\chi(\gamma)$, gives
\be\label{eq:opposite-twisted-series}
    \sum_{\gamma\in\Gamma}
    e^{-\chi(\gamma)}e^{-s d_\varphi(e,\gamma^{-1})}
    =
    \sum_{\gamma\in\Gamma}
    e^{-\chi(\gamma)}e^{-s d_{\varphi\circ\op{i}}(e,\gamma)}
    <\infty.\ee 
Combining these estimates and applying the arithmetic-geometric
mean inequality, we get
$$
\begin{aligned}
\infty
&>
\sum_{\gamma\in\Gamma}
\left(
e^{\chi(\gamma)}e^{-s d_\varphi(e,\gamma)}
+
e^{-\chi(\gamma)}e^{-s d_{\varphi\circ\op{i}}(e,\gamma)}
\right)  \\
&\ge
2\sum_{\gamma\in\Gamma}
\exp\left(
-\frac{s}{2}
\bigl(d_\varphi(e,\gamma)+d_{\varphi\circ\op{i}}(e,\gamma)\bigr)
\right) =
2\sum_{\gamma\in\Gamma}
e^{-s d_{\bar\varphi}(e,\gamma)}.
\end{aligned}
$$
This proves \eqref{eq:sym-series}.
\end{proof}

\begin{Rmk}\label{general} The key ingredient for the proof of Theorem \ref{coa} is  Proposition \ref{t3}.
This is  known to be true for relatively Morse subgroups and for linear forms which are positive on the Morse limit cone \cite{DKO_AR}. For instance, our results apply to cusped Hitchin subgroups.
They also apply to
$\theta$-Anosov and $\theta$-Morse subgroups as well, provided we replace the corresponding critical exponents and growth indicators by their $\theta$-versions as introduced in \cite{KOW_indicators}.
    \end{Rmk}

\section{Surviving rigidity for growth indicators and Riemannian exponents}
In this section, we derive Theorems \ref{fd} and \ref{m3} from the
symmetrized critical-exponent inequality, Theorem \ref{coa}.
Let $G$ be a connected semisimple real algebraic group. For a discrete subgroup $\Ga<G$ and
for any vector $u\in \fa^+$, set
\begin{equation}\label{grow}
\psi_{\Gamma}(u):=\|u\| \inf_{\underset{u \in{\mathcal C}}{\mathrm{open\;cones\;}{\mathcal C}\subset \fa^+}}
\tau_{\cal C}
\end{equation}
where $\tau_{\cal C}$ is the abscissa of convergence of the series $\sum_{\ga\in\Ga,\,\mu(\ga)\in\cal C}e^{-t\norm{\mu(\ga)}}$. This definition is independent of the choice of a norm on $\fa$. Set $\psi_{\Gamma}(0)=0$.
The resulting function $\psi_{\Gamma}$ on $\fa^+$ is the growth indicator of $\Ga$, which coincides with the one given in \eqref{gr} for non-elementary discrete subgroups. Quint \cite{Quint02} showed that $\psi_{\Gamma}$ is concave, upper-semicontinuous, and its support is precisely the limit cone:
$$\L_\Gamma= \{u\in \fa^+: \psi_{\Gamma}(u)\ge 0\}.$$
Moreover, $ {\psi_{\Gamma}}$ is  positive on $\op{int}\L_\Gamma$. We have $\psi_{\Gamma} \circ \op{i}=\psi_{\Gamma}$.

\subsection*{Rigidity on opposition-fixed directions} 
For the rest of this section, let $\Gamma<G$ be a Zariski-dense
Borel-Anosov subgroup, and let $N\lhd\Gamma$ be a coamenable normal
subgroup. Then $N$ is necessarily Zariski dense.

We use the following standard properties of Zariski-dense Borel-Anosov
subgroups.
\begin{theorem}\label{st}
We have
$\mathcal L_\Gamma-\{0\}\subset \operatorname{int}\mathfrak a^+$ and $\psi_\Gamma$ is strictly concave on
$\operatorname{int}\mathcal L_\Gamma$.
\end{theorem}

The first assertion follows from \cite{KLP17}; the strict concavity follows
from \cite[Section 4]{Quint_indicator} and \cite{sambarino2022report}.

We first prove the growth-indicator statement.

\GrowthIndicatorRigidity*

\begin{proof} By \Cref{lem:symmetric-directions-normality}, $\L_N^{\i}=\L_\Ga^{\i}$.
Since $\psi_{N}=-\infty$ outside $\L_N$, it suffices to show that $\psi_N(v)=\psi_\Ga(v)$ for all $ v\in \mathcal L_N$ and
   $ \op{i}(v)=v$. First consider the case where $v\in\inte\L_\Ga$. Let $\varphi\in\mathfrak a^*$ be a linear form tangent to
$\psi_{\Gamma}$ at $v$, that is, $\varphi\ge \psi_{\Gamma}$ and
$\varphi(v)=\psi_{\Gamma}(v)$. Such a form exists since
$v\in \operatorname{int}\mathcal L_\Gamma$ and $\psi_{\Gamma}$ is concave (cf. \cite[Lemma 2.24]{ELO_anosov}).
  Since $\psi_{\Gamma}$ is invariant under $\op{i}$ and $\op{i}(v)=v$, the linear form
$\varphi\circ\op{i}$ is also tangent to $\psi_{\Gamma}$ at $v$.  Hence replacing $\varphi$ with $\bar{\varphi}$,
we may assume that $\varphi=\varphi\circ\op{i}$. Moreover, since $\Gamma$ is Borel-Anosov, a tangent form to
$\psi_{\Gamma}$ at an interior point of $\mathcal L_\Gamma$ is positive on
$\mathcal L_\Gamma-\{0\}$. 
Because $\varphi$ is tangent to $\psi_{\Gamma}$, we have $\delta_{\Gamma,\varphi}=1$.
Since $\varphi$ is symmetric, Theorem \ref{coa} gives
$$\delta_{N,\varphi}=\delta_{\Gamma,\varphi}=1.$$

Since $\varphi >0$ on $\L_N-\{0\}$ and $\delta_{N,\varphi}=1$, it follows that $\varphi$ is tangent to $\psi_{N}$ at some $u\in \L_N$ (see \cite[Corollary 4.6]{KOW_indicators}).
Now
$$
    \varphi(u)=\psi_{N}(u)\le \psi_{\Gamma}(u)\le \varphi(u),
$$
and hence $\psi_{\Gamma}(u)=\varphi(u)$. 
We claim that $u$ lies on the ray $\mathbb R_{>0}v$. Suppose not. For
$0<t<1$, set
$$
    w_t=(1-t)v+tu.
$$
Since $v\in\operatorname{int}\mathcal L_\Gamma$ and
$u\in\mathcal L_\Gamma$, we have
$w_t\in\operatorname{int}\mathcal L_\Gamma$. By concavity of
$\psi_{\Gamma}$ and the inequalities $\psi_{\Gamma}\le\varphi$, we get
$$
\psi_{\Gamma}(w_t)
\ge
(1-t)\psi_{\Gamma}(v)+t\psi_{\Gamma}(u)
=
(1-t)\varphi(v)+t\varphi(u)
=
\varphi(w_t)
\ge
\psi_{\Gamma}(w_t).
$$
Therefore equality holds throughout. Hence $\psi_{\Gamma}$ is affine on the
segment joining $v$ to $u$, whose interior lies in
$\operatorname{int}\mathcal L_\Gamma$. Since $\psi_{\Gamma}$ is strictly
concave on $\operatorname{int}\mathcal L_\Gamma$, this is possible only if
$u$ and $v$ lie on the same ray. Thus $u=tv$ for some $t>0$.
 By homogeneity, we get $\psi_{N}(v)=\psi_{\Gamma}(v)$, as desired. 
 
We now extend this to the case when $v$ lies in the boundary of $\L_\Gamma$. Since $\L_N$ has nonempty interior and $\i$-invariant, we can choose $w\in \inte \L_N$ with $\i(w)=w$. For $0<t\le 1$,
 define $v_t=(1-t)v+tw$. Then for all $0<t<1$,
 $\i(v_t)=v_t$ and $v_t\in \inte \L_N\subset \inte \L_\Ga$,  and hence $\psi_N(v_t)=\psi_\Ga(v_t)$.
 We use the elementary fact that a finite
concave upper-semicontinuous function on a compact interval is continuous. Applying this to the restrictions of $\psi_N$ and $\psi_\Ga$ to the line segment
$\{v_t: 0\le t\le 1\}$, we get $\psi_N(v)=\psi_\Ga(v)$.
\end{proof}

The hypothesis that $v$ is symmetric cannot be dropped: see
\Cref{thm:growth-indicator-gap}.

\subsection*{Riemannian critical exponents}
Consider the associated Riemannian symmetric space $(G/K, d)$ and set $o = [K]$. 
Fix a $K$-invariant inner product $\langle \cdot, \cdot \rangle$ on $\fg$  induced from the Killing form on $\fg$ so that the induced norm $\|\cdot \|$ satisfies $$d(go, ho)=\|\mu(g^{-1}h)\|\quad\text{ for any $g, h \in G$.}$$

 Since $\psi_{\Gamma}$ is concave, upper-semicontinuous, and the unit norm ball in $\fa$ is strictly convex, there exists a unique unit vector
  ${{u_\Ga}}\in 
  \cal L_\Gamma$ (called the maximal growth direction) such that the critical exponent $\delta_\Ga$ with respect to the Riemannian distance $d$ satisfies:
   \begin{equation}\label{tug}\delta_\Ga=\max_{u\in\mathfrak a^+,\norm{u}=1}\psi_{\Gamma}(u)=\psi_{\Gamma}({u_\Ga}).\end{equation}
We have $\op{i}(u_{\Gamma})=u_\Gamma$.
\RiemannianCriticalExponentRigidity*

\begin{proof}
Let $v_N\in\mathcal L_N$ be the maximal growth direction of $N$.  By Quint's theorem
\cite[Corollary 4.2.4]{Quint02}, the linear form
$$
    \varphi_{N}(u)\coloneqq\langle u,v_N\rangle
$$
satisfies $$\delta_{N,\varphi_{N}}=\delta_N.$$ 
Since  $\Ga$ is Borel-Anosov, we have $\L_\Ga-\{0\}\subset \inte \fa^+$ and hence 
$v_N\in\operatorname{int}\mathfrak a^+$. Since $\fa^+$ is acute with respect to the inner product $\langle \cdot, \cdot\rangle$, 
$\varphi_{N}$ is positive on $\cal L_{\Ga}-\{0\}$. Since both $\psi_N$ and $v_N$
are invariant under $\i$, the uniqueness of the maximal growth direction implies  $\op{i}(v_N)=v_N$ and hence $\varphi_N$ is symmetric.
Therefore Theorem \ref{coa} gives
\be\label{pre}\delta_{\Gamma,\varphi_{N}}=\delta_{N,\varphi_{N}}=\delta_N .\ee 
On the other hand, since $\|v_N\|=1$, the Cauchy--Schwarz inequality gives
$$
    \varphi_{N}(\mu(\gamma))=\langle \mu(\gamma),v_N\rangle
    \le \|\mu(\gamma)\|
    \quad\text{ for all }\gamma\in\Gamma.
$$
Hence
$$
    \{\gamma\in\Gamma:\|\mu(\gamma)\|<T\}
    \subset
    \{\gamma\in\Gamma:\varphi_{N}(\mu(\gamma))<T\},
$$
and therefore $$\delta_\Gamma\le \delta_{\Gamma,\varphi_{N}}.$$
Combining this with \eqref{pre} gives $\delta_\Gamma\le \delta_N$.  Hence $\delta_N=\delta_\Gamma$.
\end{proof}

Indeed, the same proof works for the following: if $\|\cdot\|$ is any norm on $\fa$ which is $\i$-invariant and $\delta_{\star, \|\cdot\|}:=\limsup_{T\to \infty}\frac{1}{T}\log \#\{g\in \star: \|\mu(g)\|<T\}$,  then $$\delta_{\Ga, \|\cdot\|}=\delta_{N, \|\cdot\|}.$$

\section{Sharpness: equal limit cones and different growth indicators}

Theorem \ref{fd} identifies the opposition-fixed directions as a
universal rigidity locus. We now show that this is optimal in the strongest
relevant sense: away from that locus, directional critical exponents may
differ, and the growth indicators may differ at an interior direction even
when the two limit cones coincide.

Let $G=\SL_3(\mathbb R)$. We denote by $(\fa^+)^\vee$ the closed dual
cone of $\fa^+$, namely
$$
    (\fa^+)^\vee
    =
    \{\varphi\in\fa^*:\varphi(v)\ge 0 \text{ for all } v\in\fa^+\}.
$$

We prove two sharpness results. First, for every fixed nonsymmetric
$\varphi\in(\fa^+)^\vee-\{0\}$, there is a Zariski dense Borel-Anosov group
$\Gamma<\SL_3(\mathbb R)$ with a cocyclic Zariski dense normal subgroup
$N\lhd\Gamma$ such that
\be\label{counter}
    \delta_{N,\varphi}<\delta_{\Gamma,\varphi}.
\ee 
Second, there is such a pair $N\lhd\Gamma$ for which the growth
indicators differ at a nonsymmetric interior direction of the normal
subgroup's limit cone.

\medskip

 The following pressure lemma is the key
input for producing a strict gap for nonsymmetric linear forms. By a Schottky representation
$j:F_2=\langle a\rangle *\langle b\rangle\to\SL_2(\mathbb R)$, we mean a faithful
(convex cocompact) representation such that $j(a)$ and $j(b)$ form a
classical ping-pong pair on $\mathbb P^1(\mathbb R)$: there are pairwise
disjoint intervals $I_a^\pm,I_b^\pm$ around their attracting and repelling
fixed points such that
$$    j(a)^{\pm1}\bigl(\mathbb P^1(\mathbb R)- I_a^{\mp}\bigr)
    \subset I_a^\pm,
    \qquad
    j(b)^{\pm1}\bigl(\mathbb P^1(\mathbb R)- I_b^{\mp}\bigr)
    \subset I_b^\pm .$$

\begin{lemma}\label{lem:pressure-gap}
Let $j:F_2\to\SL_2(\mathbb R)$ be a Schottky representation, and let
$\ell:F_2\to\mathbb R_{\ge0}$ be defined by
\be\label{defl}
    \mu(j(\gamma))=\operatorname{diag}(\ell(\gamma),-\ell(\gamma))
    \qquad(\gamma\in F_2).
\ee 
Let $\chi_0:F_2\to\mathbb R$ be a non-trivial homomorphism. Then, for all
sufficiently small $t\ne0$,
$$
    \delta_{F_2,\ell+t\chi_0}
    >
    \delta_{F_2,\ell}.
$$
Here, for a proper function $f:F_2\to\mathbb R$,
$$
    \delta_{F_2,f}
    :=
    \limsup_{T\to\infty}
    \frac1T
    \log\#\{\gamma\in F_2:f(\gamma)<T\}.
$$
\end{lemma}

\begin{proof}
Let $S=\{a^{\pm1},b^{\pm1}\}$ be the free generating set of $F_2$, and let
$$
    \Sigma=\{(x_i)_{i\ge 0}\in S^{\mathbb N}: x_{i+1}\ne x_i^{-1}\}
$$
be the one-sided subshift coding reduced words. We denote by $\sigma$ the
shift map. By Quint's Schottky coding for orbital counting
\cite[Sections 4--5]{Quint_Schottky}, 
there exists a positive H\"older function $L:\Sigma\to\mathbb R$ with the
following property: there is a constant $C>0$ such that, for every reduced
word $\gamma=s_0s_1\cdots s_{n-1}\in F_2$
and every $x\in\Sigma$ beginning with the block $s_0s_1\cdots s_{n-1}$,
we have
$$
    \left|
    \sum_{k=0}^{n-1} L(\sigma^k x)-\ell(\gamma)
    \right|\le C.
$$

Since $\chi_0:F_2\to\mathbb R$ is a homomorphism, it is represented by the
locally constant function
$$
    \mathcal X(x)=\chi_0(x_0),\qquad x=(x_i)_{i\ge0}\in\Sigma.
$$
Thus, if $x$ begins with the reduced word $\gamma=s_0\cdots s_{n-1}$, then
$$
    \sum_{k=0}^{n-1}\mathcal X(\sigma^k x)=\chi_0(\gamma).
$$
Therefore, for every $t\in\mathbb R$,
$$
    \left|
    \sum_{k=0}^{n-1}(L+t\mathcal X)(\sigma^k x)
    -
    \bigl(\ell(\gamma)+t\chi_0(\gamma)\bigr)
    \right|
    \le C.
$$

Since $L$ is positive and continuous on the compact space $\Sigma$, there
is $c_0>0$ such that $L\ge c_0$. Hence, for all sufficiently small
$|t|$, the function
$$
    L_t:=L+t\mathcal X
$$
is still positive. In particular, $\ell+t\chi_0$ is proper on $F_2$, and
the bounded error above does not change the exponential growth exponent.
Thus
$$
    H(t):=\delta_{F_2,\ell+t\chi_0}
$$
is characterized by the pressure equation
$$
    p(-H(t)L_t)=0,
$$
where $p$ denotes topological pressure for $(\Sigma,\sigma)$.

The map $(s,t)\mapsto p(-sL_t)$ is real analytic. Moreover,
$$
    \frac{\partial}{\partial s}p(-sL_t)
    =
    -\int L_t\,dm_{s,t}<0,
$$
where $m_{s,t}$ is the equilibrium state of $-sL_t$. Hence the implicit
function theorem implies that $H(t)$ is real analytic for $|t|$ small.

We next observe that $H$ is even. Since
$\ell(\gamma^{-1})=\ell(\gamma)$ and $
    \chi_0(\gamma^{-1})=-\chi_0(\gamma)$,
the change of variables $\gamma\mapsto\gamma^{-1}$ gives
$$
    \#\{\gamma\in F_2:\ell(\gamma)+t\chi_0(\gamma)<T\}
    =
    \#\{\gamma\in F_2:\ell(\gamma)-t\chi_0(\gamma)<T\}.
$$
Therefore
$ H(t)=H(-t)$  and hence $ H'(0)=0.$

Let $h=H(0)=\delta_{F_2,\ell}$, and let $m$ be the equilibrium state for
the potential $-hL$. Differentiating
$ p\bigl(-H(t)(L+t\mathcal X)\bigr)=0$
at $t=0$, and using the pressure derivative formula, gives
$$
    0
    =
    \int \bigl(-H'(0)L-h\mathcal X\bigr)\,dm.
$$
Since $H'(0)=0$, we obtain
 $ \int \mathcal X\,dm=0$. Differentiating a second time and using the standard pressure Hessian formula
gives
$$
    0
    =
    -H''(0)\int L\,dm
    +
    h^2\operatorname{Var}_m(\mathcal X),
$$
where $\operatorname{Var}_m(\mathcal X)$ is the asymptotic variance of
$\mathcal X$ with respect to the equilibrium state $m$. Hence
$$
    H''(0)\int L\,dm
    =
    h^2\operatorname{Var}_m(\mathcal X).
$$
Since $L>0$, we have $\int L\,dm>0$. It remains to show that
$\operatorname{Var}_m(\mathcal X)>0$.

By the Livsic degeneracy criterion for a topologically mixing subshift of
finite type, $\operatorname{Var}_m(\mathcal X)=0$ if and only if
$\mathcal X$ is cohomologous to a constant. Suppose, for contradiction,
that
$$
    \mathcal X=c+u-u\circ\sigma
$$
for some continuous function $u$ and some constant $c$. Then for every
periodic word $w$ of period $n$,
$$
    \chi_0(w)=\sum_{k=0}^{n-1}\mathcal X(\sigma^k x)=cn.
$$
Applying this to $w^{-1}$, which has the same period $n$, gives
$ \chi_0(w)=-cn.$
Thus $\chi_0(w)=0$ for every cyclically reduced word $w$. This is
impossible because $\chi_0$ is a non-trivial homomorphism on $F_2$. Therefore
$$
    \operatorname{Var}_m(\mathcal X)>0\quad \text{ and hence}\quad
    H''(0)>0.
$$

Since $H$ is real analytic, even, and satisfies $H''(0)>0$, the point
$t=0$ is a strict local minimum of $H$. Hence, for all sufficiently small
$t\ne0$,
$$
    \delta_{F_2,\ell+t\chi_0}>
    \delta_{F_2,\ell}.
$$
This proves the lemma.
\end{proof}

\subsection*{The model representation}
For $v=\operatorname{diag}(v_1,v_2,v_3)\in\mathfrak a^+$ for
$\SL_3(\mathbb R)$, define
$$
    \omega_1(v):=v_1 \quad\text{ and } \quad \omega_2(v):=v_1+v_2.
$$
These are the fundamental weights.

In the rest of this section, fix a linear form
$$
    \varphi\in(\fa^+)^\vee-\{0\},
    \qquad
    \varphi=s_1\omega_1+s_2\omega_2,
    \qquad
    s_1,s_2\ge0,\quad s_1+s_2>0.
$$
Then $\varphi$ is nonsymmetric if and only if $s_1\ne s_2$.

In order to construct an example satisfying \eqref{counter}, we begin with a  Schottky representation
$$j:F_2= \langle a\rangle *\langle b\rangle\to\SL_2(\mathbb R) $$ and  the function $\ell:F_2\to \br_{\ge 0}$ defined in \eqref{defl}.
Let $\chi:F_2 \rightarrow \mathbb{R}$ be the homomorphism with $\chi(a)=1$ and $\chi(b)=0$, and
set
$$
    \mathcal N:=\ker{\chi}=\langle\!\langle b\rangle\!\rangle_{F_2}.
$$

For $\kappa\ge 0$, define
$\rho_{\kappa}:F_2\to\SL_3(\mathbb R)$
by
\be\label{rhokappa}
    \rho_{\kappa}(\gamma)
    =
    \begin{pmatrix}
        e^{-\kappa{\chi}(\gamma)/2}j(\gamma)&0\\
        0&e^{\kappa{\chi}(\gamma)}
    \end{pmatrix}.
\ee

Set
$$
    \Gamma_{\kappa}:=\rho_{\kappa}(F_2),
    \qquad
    N_{\kappa}:=\rho_{\kappa}(\mathcal N).
$$

\begin{theorem}[Model representation: all nonsymmetric positive forms]
\label{thm:model-all-positive-forms} For all sufficiently small $\kappa>0$, we have
$$
    \delta_{\Gamma_{\kappa},\varphi}
    =
    \frac{1}{s_1+s_2}\,
    \delta_{F_2,\ell+\frac{\kappa(s_2-s_1)}{2(s_1+s_2)}{\chi}}\qquad \text{and}\qquad 
    \delta_{N_{\kappa},\varphi}
    =
    \frac{1}{s_1+s_2}\,\delta_{F_2,\ell}.
$$
Consequently,  $\varphi$ is nonsymmetric if and only if
$$\delta_{N_{\kappa},\varphi}   <\delta_{\Gamma_{\kappa},\varphi}.$$

\end{theorem}

\begin{proof}
By Lemma \ref{lem:pressure-gap}, applied to $j$ and ${\chi}$, there
exists $\eta>0$ such that
\begin{equation}\label{eq:pressure-gap-all-positive-forms}
    \delta_{F_2,\ell+t{\chi}}
    >
    \delta_{F_2,\ell} \quad\text{for all $0<|t|<\eta$.}
\end{equation}
Choose $\kappa>0$ sufficiently small so that $\kappa/2<\eta$ and, for
all $\gamma\in F_2-\{e\}$,
\begin{equation}\label{eq:kappa-choice-all-positive-forms}
    \ell(\gamma)>\frac{3\kappa}{2}|{\chi}(\gamma)|.
\end{equation}
This is possible because $j$ is convex cocompact, so $\ell(\gamma)$
grows linearly in the word length, while $|{\chi}(\gamma)|$ is bounded
above linearly in the word length.
The choice of $\kappa$ ensures that, for every non-trivial
$\gamma\in F_2$, the Cartan projection of $\rho_{\kappa}(\gamma)$ is
$$
    \mu(\rho_{\kappa}(\gamma))
    =
    \operatorname{diag}\left(
        \ell(\gamma)-\frac{\kappa}{2}{\chi}(\gamma),\,
        \kappa{\chi}(\gamma),\,
        -\ell(\gamma)-\frac{\kappa}{2}{\chi}(\gamma)
    \right).
$$

Hence, for all $\gamma\in F_2-\{e\}$,
$$
    \varphi(\mu(\rho_{\kappa}(\gamma)))
    = (s_1+s_2)
    \left(
        \ell(\gamma)
        +   t_{\varphi} {\chi}(\gamma)
    \right)
$$ where $
    t_{\varphi}:=
    \frac{\kappa(s_2-s_1)}{2(s_1+s_2)}$.
Thus 
we obtain
$$
    \delta_{\Gamma_{\kappa},\varphi}
    =
   \tfrac{1}{s_1+s_2}\,
    \delta_{F_2,\ell+t_{\varphi}{\chi}}.
$$

On the other hand, ${\chi}=0$ on $\mathcal N$. Hence, for every
$h\in\mathcal N-\{e\}$,
$\varphi(\mu(\rho_{\kappa}(h))) = (s_1+s_2)\ell(h).$
Therefore
$$
    \delta_{N_{\kappa},\varphi}
    =
 \tfrac{1}{s_1+s_2}\,\delta_{\mathcal N,\ell}.
$$
Since $j(\mathcal N)$ is cocyclic in the convex cocompact group
$j(F_2)<\SL_2(\mathbb R)$, Roblin's rank-one theorem
\cite{Roblin05} gives
$
    \delta_{\mathcal N,\ell}
    =
    \delta_{F_2,\ell}.
$
Thus
$$
    \delta_{N_{\kappa},\varphi}
    =
     \tfrac{1}{s_1+s_2}\,\delta_{F_2,\ell}.
$$

Since $|s_2-s_1|\le s_1+s_2$, we have $|t_{\varphi}|\le \kappa/2<\eta$. 
If $s_1\ne s_2$, then $t_{\varphi}\ne0$, and hence
\eqref{eq:pressure-gap-all-positive-forms} gives
$
    \delta_{F_2,\ell+t_{\varphi}{\chi}}
    >
    \delta_{F_2,\ell}.
$
Therefore $
    \delta_{\Gamma_{\kappa},\varphi} >
    \delta_{N_{\kappa},\varphi}.$
If $s_1=s_2$, then $t_{\varphi}=0$, so the two displayed formulae give
$
    \delta_{\Gamma_{\kappa},\varphi}
    =
    \delta_{N_{\kappa},\varphi}.
$
This completes the proof.
\end{proof}

\subsection*{Adding a Schottky generator}
The normal subgroup $N_{\kappa}$  has only the barycentric direction in its limit cone; in particular, it contains no nonsymmetric interior direction.  To prove the sharpness of the symmetry
condition for growth indicators, we need a nonsymmetric direction lying in the interior of the limit cone of the normal subgroup.  We therefore add a third Schottky
generator $c\mapsto {\mathsf c}$ so that limit cone of the normal subgroup contains the convex cone spanned by two rays containing $\lambda({\mathsf c})$ and $\lambda({\mathsf c}^{-1})$. 
Choosing $\lambda({\mathsf c})$ not fixed by $\i$ gives
nonsymmetric directions in the interior of this cone. This will allow us
to locate the maximizing direction for a nearby nonsymmetric linear form
inside the limit cone of the normal subgroup.

More precisely, fix a letter $c$ and
let $$F_3=F_2*\langle c\rangle.$$

Let $\kappa\ge 0$ be small enough to satisfy \eqref{choice} and recall the representation $\rho_\kappa: F_2\to \SL_3(\br)$ from \eqref{rhokappa}. 
For a loxodromic element
${\mathsf c}\in\SL_3(\mathbb R)$ and $r\in \mathbb N$, define
a representation
\begin{equation}\label{eqn:F3}
        \widehat\rho_{\cc, \kappa ,r}:F_3\to\SL_3(\mathbb R)
\end{equation}
by
$$
   \widehat\rho_{\cc, \kappa,r}|_{F_2}=\rho_{\kappa},\qquad
    \widehat\rho_{\cc, \kappa,r}(c)={\mathsf c}^r.
$$
We choose $\mathsf c\in\SL_3(\mathbb R)$ loxodromic and in Schottky position with
both $\Gamma_0$ and $\rho_\kappa(F_2)$. We choose $\mathsf c$ generically so
that for every $r\ge1$,
 $\widehat\rho_{\mathsf c, \kappa ,r}$
 is Zariski dense. This is possible by \cite[Proposition 4.4]{Tits_free}.
 For all sufficiently large $r$,
the representations $\widehat\rho_{\cc, \kappa,r}$ are then Borel-Anosov by
the combination theorem \cite{DKL19}. For simplicity, we omit the subscript $\mathsf c$ from the notation and write
$$\widehat\rho_{\kappa ,r}:=\widehat\rho_{\cc, \kappa ,r}.$$

Consider the homomorphism:
$$
    \cal X:F_3\to\mathbb Z,
    \qquad
    {\mathcal X}(a)=1,\quad {\mathcal X}(b)=0,\quad {\mathcal X}(c)=0,
$$
and set
$$
    \mathcal N':=\ker{\mathcal X}.
$$
Equivalently, $\mathcal N'$ is the normal closure of $\langle b,c\rangle$
in $F_3$, and $F_3/\mathcal N'\simeq\mathbb Z$.

  Recall $\Gamma_{\kappa}=\rho_{\kappa}(F_2)$ and set 
$$\Gamma_{\kappa,r}:=\widehat\rho_{\kappa,r}(F_3) \quad\text{and}\quad N'_{\kappa,r}:=\widehat\rho_{\kappa,r}(\cal N').
$$

\begin{lemma}[Free-product estimates]\label{lem:free-product-estimates}
Let $\mathcal U\subset(\fa^+)^\vee-\{0\}$ be compact, and let
$\kappa\ge0$ be sufficiently small. There exist positive constants
$a_{\mathcal U}, B_{\mathcal U}, A_{\mathcal U}$ and $r_0\ge 1$
such that the following hold.

\begin{enumerate}
    \item For all $\varphi\in\mathcal U$, $r\ge1$, and
    $q\in\mathbb Z-\{0\}$,
    \begin{equation}\label{eq:c-block-linear-growth}
        \varphi(\mu(\mathsf c^{rq}))
        \ge
        a_{\mathcal U}r|q|-B_{\mathcal U}.
    \end{equation}

    \item For every reduced word $\gamma=g_0c^{q_1}g_1\cdots c^{q_k}g_k$
with  $g_i\in F_2$ and $ q_i\in\mathbb Z-\{0\}$,
all  $r\ge r_0$ and for all $\varphi\in\mathcal U$,
    \begin{multline}\label{eq:free-product-lower}
   \sum_{i=0}^{k}\varphi(\mu(\rho_{\kappa}(g_i)))
    +
    \sum_{i=1}^{k}\varphi(\mu(\mathsf c^{rq_i}))
    -
    A_{\mathcal U}(k+1)\le \varphi(\mu(\widehat\rho_{\kappa,r}(\gamma)))
   \\  \le   \sum_{i=0}^{k}\varphi(\mu(\rho_{\kappa}(g_i)))
    +
    \sum_{i=1}^{k}\varphi(\mu(\mathsf c^{rq_i}))
   \end{multline}

\end{enumerate}
\end{lemma}

\begin{proof}
Since $\mathsf c$ is loxodromic, both $\lambda(\mathsf c)$ and
$\lambda(\mathsf c^{-1})$ lie in $\operatorname{int}\fa^+$. Compactness
of $\mathcal U\subset(\fa^+)^\vee-\{0\}$, together with
$$
    n^{-1}\mu(\mathsf c^n)\to\lambda(\mathsf c),
    \qquad
    n^{-1}\mu(\mathsf c^{-n})\to\lambda(\mathsf c^{-1}),
$$
gives \eqref{eq:c-block-linear-growth}. For \eqref{eq:free-product-lower}, apply the usual Schottky ping-pong
estimate in the two fundamental representations $\wedge^j\mathbb R^3$,
$j=1,2$ (see \cite{Benoist97}). Taking logarithms and then applying $\varphi$
gives the stated lower bound, uniformly for $\varphi\in\mathcal U$. The
upper estimate  follows directly from
submultiplicativity in the same two fundamental representations.
\end{proof}

\begin{proposition}\label{prop:add-generator-convergence}
Let $\mathcal U\subset(\fa^+)^\vee-\{0\}$ be compact. For all
sufficiently small $\kappa\ge0$, we have
$$
  \lim_{r\to\infty}
  \sup_{\varphi\in\mathcal U}
    \left|
    \delta_{\Gamma_{\kappa,r},\varphi}
    -
    \delta_{\rho_{\kappa}(F_2),\varphi}
    \right|=0.
$$
\end{proposition}

\begin{proof}
Since $\widehat\rho_{\kappa,r}|_{F_2}=\rho_\kappa$, we have
$\delta_{\Gamma_{\kappa,r},\varphi}
    \ge
    \delta_{\rho_\kappa(F_2),\varphi}$
for all $r$ and all $\varphi\in\mathcal U$.

Fix $\epsilon>0$, and put
$$
    s_\varphi:=\delta_{\rho_\kappa(F_2),\varphi}+\epsilon.
$$
Since $\rho_\kappa(F_2)$ is Borel-Anosov and
$\mathcal U\subset(\fa^+)^\vee-\{0\}$ is compact, the pressure
formalism for Schottky groups implies that $s_\varphi$ is bounded above
and bounded away from zero on $\mathcal U$. Moreover,
$$
    S_F(\varphi)
    :=
    \sum_{g\in F_2}
    e^{-s_\varphi\varphi(\mu(\rho_\kappa(g)))}
$$
is uniformly bounded for $\varphi\in\mathcal U$. By \eqref{eq:c-block-linear-growth},
$$
    S_{\mathsf c}(r,\varphi)
    :=
    \sum_{q\in\mathbb Z-\{0\}}
    e^{-s_\varphi\varphi(\mu(\mathsf c^{rq}))}
    \to0 \quad\text{uniformly for $\varphi\in\mathcal U$}.
$$

Now let $\gamma=g_0c^{q_1}g_1\cdots c^{q_k}g_k$
be reduced. By \eqref{eq:free-product-lower},
$$
\begin{aligned}
&\sum_{\gamma\in F_3}
e^{-s_\varphi\varphi(\mu(\widehat\rho_{\kappa,r}(\gamma)))} \le
e^{s_\varphi A_{\mathcal U}}S_F(\varphi)
\sum_{k\ge0}
\left(
e^{s_\varphi A_{\mathcal U}}
S_F(\varphi)S_{\mathsf c}(r,\varphi)
\right)^k .
\end{aligned}
$$
Since $s_\varphi$ and $S_F(\varphi)$ are uniformly bounded on
$\mathcal U$, while $S_{\mathsf c}(r,\varphi)\to0$ uniformly, the
geometric series converges for all sufficiently large $r$, uniformly in
$\varphi\in\mathcal U$. Hence
$$
    \delta_{\Gamma_{\kappa,r},\varphi}
    \le
    \delta_{\rho_\kappa(F_2),\varphi}+\epsilon
$$
for all $\varphi\in\mathcal U$ and all sufficiently large $r$. Combining
this with the lower bound and letting $\epsilon\to0$ proves the
proposition.
\end{proof}

\begin{lemma}\label{lem:normal-upper-F3-general}
For every $\varphi\in(\fa^+)^\vee-\{0\}$ and all sufficiently small
$\kappa\ge0$, we have
$$
    \limsup_{r\to\infty}
    \delta_{N'_{\kappa,r},\varphi}
    \le
    \delta_{\Gamma_0,\varphi}.
$$
\end{lemma}

\begin{proof}
Apply Lemma \ref{lem:free-product-estimates} with
$\mathcal U=\{\varphi\}$, and write the resulting constant as
$A_\varphi$. Let
$w=w_0c^{q_1}w_1\cdots c^{q_k}w_k\in\mathcal N'$ with
   $ w_i\in F_2$ and $ q_i\in\mathbb Z-\{0\}$.
Since $w\in\mathcal N'$, the sum of the $a$-exponents of the
$F_2$-blocks is zero: $\sum_{i=0}^{k}\mathcal X(w_i)=0$.

For $u\in F_2-\{e\}$, by the definition of $\rho_{\kappa}$ and the
choice of $\kappa$, we have
$$
    \mu(\rho_{\kappa}(u))
    =
    \operatorname{diag}\left(
        \ell(u)-\frac{\kappa}{2}{\mathcal X}(u),\,
        \kappa{\mathcal X}(u),\,
        -\ell(u)-\frac{\kappa}{2}{\mathcal X}(u)
    \right).
$$
Hence
$$
    \omega_1(\mu(\rho_{\kappa}(u)))
    =
    \ell(u)-\frac{\kappa}{2}{\mathcal X}(u)\text{ and }
    \omega_2(\mu(\rho_{\kappa}(u)))
    =
    \ell(u)+\frac{\kappa}{2}{\mathcal X}(u).
$$
Therefore
$$
    \varphi(\mu(\rho_{\kappa}(u)))
    =
    \varphi(\mu(\rho_{0}(u)))
    +
    \frac{\kappa}{2}(s_2-s_1){\mathcal X}(u).
$$
Applying this to the $F_2$-blocks $w_i$, and using
$\sum_i{\mathcal X}(w_i)=0$, gives
\begin{equation}\label{eq:F2-block-cancellation-general}
    \sum_{i=0}^{k}
    \varphi(\mu(\rho_{\kappa}(w_i)))
    =
    \sum_{i=0}^{k}
    \varphi(\mu(\rho_{0}(w_i))).
\end{equation}

Let $r'=\lfloor r/2\rfloor$. By \eqref{eq:c-block-linear-growth}, after
increasing $r$, we may assume that, for every $q\ne0$,
\begin{equation}\label{eq:C-comparison-varphi}
    \varphi(\mu(\mathsf c^{rq}))
    \ge
    \varphi(\mu(\mathsf c^{r'q}))+A_\varphi .
\end{equation}
Combining \eqref{eq:free-product-lower},
\eqref{eq:F2-block-cancellation-general}, and
\eqref{eq:C-comparison-varphi}, we get
$$
\begin{aligned}
    \varphi(\mu(\widehat\rho_{\kappa,r}(w)))
    \ge&
    \sum_{i=0}^{k}
    \varphi(\mu(\rho_0(w_i)))
    +
    \sum_{i=1}^{k}
    \varphi(\mu(\mathsf c^{r'q_i}))
    -
    A_\varphi \\\ge &  \varphi(\mu(\widehat\rho_{0,r'}(w)))-A_\varphi
    \qquad(w\in\mathcal N')
\end{aligned}
$$

It follows that, for every $s>0$,
$$
\begin{aligned}
    \sum_{w\in\mathcal N'}
    e^{-s\varphi(\mu(\widehat\rho_{\kappa,r}(w)))}
    &\le
    e^{sA_\varphi}
    \sum_{w\in\mathcal N'}
    e^{-s\varphi(\mu(\widehat\rho_{0,r'}(w)))} \\
    &\le
    e^{sA_\varphi}
    \sum_{\gamma\in F_3}
    e^{-s\varphi(\mu(\widehat\rho_{0,r'}(\gamma)))} .
\end{aligned}
$$
Therefore
$$
    \delta_{N'_{\kappa,r},\varphi}
    \le
    \delta_{\widehat\rho_{0,r'}(F_3),\varphi}.
$$
Using Proposition \ref{prop:add-generator-convergence} for
$\widehat\rho_{0,r'}$, we obtain
$$
    \limsup_{r\to\infty}
    \delta_{N'_{\kappa,r},\varphi}
    \le
    \delta_{\Gamma_0,\varphi}.
$$
\end{proof}

\begin{theorem}\label{thm:omega1-gap}\label{ns0}
If $\varphi\in(\fa^+)^\vee-\{0\}$ is nonsymmetric, then
there exists a Zariski dense Borel-Anosov subgroup
$\Gamma<\SL_3(\mathbb R)$ and a cocyclic Zariski dense normal subgroup
$N\lhd\Gamma$ such that
$$
    \delta_{N,\varphi}< \delta_{\Gamma,\varphi}.
$$
\end{theorem}

\begin{proof}
 Write $\varphi=s_1\omega_1+s_2\omega_2$ with $s_1,s_2\ge0$ and $s_1+s_2>0$.
By  Theorem \ref{thm:model-all-positive-forms}, together with
Lemma \ref{lem:pressure-gap},  we have
$$
    \delta_{\Gamma_{\kappa},\varphi}
    >
    \delta_{\Gamma_0,\varphi} \quad \text{ for all sufficiently small $\kappa>0$}.
$$
 Indeed, if
$
    t_\varphi
    :=
    \frac{\kappa(s_2-s_1)}{2(s_1+s_2)}\ne 0
$
is sufficiently small, then 
$$
    \delta_{\Gamma_{\kappa},\varphi}
    =
\tfrac{1}{s_1+s_2}\, 
    \delta_{F_2,\ell+t_\varphi{\chi}}
    >
   \tfrac{1}{s_1+s_2}\, 
    \delta_{F_2,\ell}
    =
    \delta_{\Gamma_0,\varphi}.
$$
By Proposition \ref{prop:add-generator-convergence}, we also have
$$ \lim_{r\to \infty} \delta_{\Gamma_{\kappa,r},\varphi}
    =
    \delta_{\Gamma_{\kappa},\varphi}.
$$
On the other hand, by Lemma \ref{lem:normal-upper-F3-general},
$$
    \limsup_{r\to\infty}
    \delta_{N'_{\kappa,r},\varphi}
    \le
    \delta_{\Gamma_0,\varphi}.
$$
Combining the last three displays gives
$$
    \limsup_{r\to\infty}
    \delta_{N'_{\kappa,r},\varphi}
    <
    \lim_{r\to\infty}
    \delta_{\Gamma_{\kappa,r},\varphi}.
$$
Hence, for all sufficiently large $r$,
$$
    \delta_{N'_{\kappa,r},\varphi}
    <
    \delta_{\Gamma_{\kappa,r},\varphi}.
$$
Fix such an $r$, and set
$$
    \Gamma:=\Gamma_{\kappa,r},
    \qquad
    N:=N'_{\kappa,r}.
$$
For $r$ sufficiently large, $\Gamma$ is Borel-Anosov and Zariski dense
by the choice of $\cc$. Since the representation is Borel-Anosov, it is
faithful. Hence
$$
    \Gamma/N\simeq F_3/\mathcal N'\simeq \mathbb Z,
$$
so $N$ is a cocyclic normal subgroup of $\Gamma$.
Finally, $N$ is Zariski dense. This proves the theorem.
\end{proof}

\subsection*{A growth-indicator gap inside a common limit cone}
We now turn to Theorem \ref{thm:growth-indicator-gap}.
Let $\kappa>0$ be sufficiently small and set for $r\in \N$,
$$
    \Gamma_r:=\widehat\rho_{\kappa,r}(F_3),
    \quad 
    N'_r:=N'_{\kappa,r}\quad
    \Gamma_\infty:=\rho_{\kappa}(F_2).
$$
Let $$\varphi_0=\omega_1+\omega_2.$$

For $r$ sufficiently large, $\Gamma_r$ is Zariski dense and
Borel-Anosov. Hence $\psi_{\Gamma_r}$ is strictly concave on
$\operatorname{int}\mathcal L_{\Gamma_r}$.
 Therefore, for every
$\varphi$ positive on $\mathcal L_{\Gamma_r}-\{0\}$, the variational
formula \eqref{var}
has a unique maximizer $v_{r,\varphi}\in
\operatorname{int}\mathcal L_{\Gamma_r}$, characterized by
$$
    \varphi(v_{r,\varphi})=1,
    \qquad
    \delta_{\Gamma_r,\varphi}=\psi_{\Gamma_r}(v_{r,\varphi}).
$$

For the model group $\Gamma_\infty$, we define $v_{\infty,\varphi}$
similarly. Its uniqueness follows from the explicit Schottky pressure
formula for $\delta_{\Gamma_\infty,\varphi}$, equivalently from the strict
convexity of the associated pressure function.

Note also that $\varphi\mapsto \delta_{\Gamma_r,\varphi}$ is convex
on $\operatorname{int}(\mathcal L_{\Gamma_r}^{\vee})=\{\varphi\in \fa^*:\varphi >0 \text{ on $\L_{\Gamma_r}-\{0\}$}\}$. Indeed, for each
fixed nonzero $v\in\mathcal L_{\Gamma_r}$, the function
$\varphi\mapsto \frac{\psi_{\Gamma_r}(v)}{\varphi(v)}$
is convex on $\operatorname{int}(\mathcal L_{\Gamma_r}^{\vee})$, since $\psi_{\Gamma_r}(v)\ge 0$ and
$x\mapsto 1/x$ is convex on $(0,\infty)$. Hence
$\varphi\mapsto \delta_{\Gamma_r,\varphi}$
is convex as the supremum of convex functions.

\begin{lemma}[Uniform convergence of maximizing directions]
\label{lem:uniform-maximizer-convergence}
There exists a neighborhood $\mathcal U$ of $\varphi_0$ such that as $r\to\infty$,
$$
    v_{r,\varphi}\to v_{\infty,\varphi} \quad\text{uniformly for $\varphi\in\mathcal U$}.
$$

\end{lemma}
\begin{proof}
Choose a compact neighborhood $\mathcal U$ of $\varphi_0$ in $\operatorname{int}((\fa^+)^\vee)$.
By Proposition \ref{prop:add-generator-convergence}, we have
$$
    \delta_{\Gamma_r,\varphi}
   \to
    \delta_{\Gamma_\infty,\varphi} \quad\text{uniformly for $\varphi\in\mathcal U$}.
$$
After shrinking $\mathcal U$ if necessary, the limiting exponent is bounded
away from zero on $\mathcal U$.
For each sufficiently large $r$, the uniqueness of the maximizer noted
above implies that $\varphi\mapsto\delta_{\Gamma_r,\varphi}$
is differentiable on $\mathcal U$. For the limiting model group, writing
$
    \varphi=s_1\omega_1+s_2\omega_2$
the explicit model formula gives
$$
    \delta_{\Gamma_\infty,\varphi}
    =
    \frac{1}{s_1+s_2}\,
    \delta_{F_2,\ell+
    \frac{\kappa(s_2-s_1)}{2(s_1+s_2)}\chi}.
$$
By the pressure argument used in the proof of
\Cref{lem:pressure-gap}, the right-hand side is real analytic for
$\varphi$ near $\varphi_0$. Hence
$$\varphi\mapsto\delta_{\Gamma_\infty,\varphi}
$$
is $C^1$ on $\mathcal U$, after shrinking $\mathcal U$ if necessary. We now use the following standard consequence of convexity: if
differentiable convex functions converge locally uniformly to a $C^1$
convex function, then their differentials converge locally uniformly on
compact subsets. Since the functions
$\varphi\mapsto \delta_{\Gamma_r,\varphi}$ are convex and converge
uniformly on $\mathcal U$ to the $C^1$ function
$\varphi\mapsto \delta_{\Gamma_\infty,\varphi}$, 
 we obtain
$$
    D_\varphi\delta_{\Gamma_r,\varphi}
    \to
    D_\varphi\delta_{\Gamma_\infty,\varphi}
    \qquad\text{uniformly for }\varphi\in\mathcal U.
$$

For every $\eta\in\fa^*$, we get
\be\label{dvv}
    D_\varphi\delta_{\Gamma_r,\varphi}(\eta)
    =
    -\delta_{\Gamma_r,\varphi}\,\eta(v_{r,\varphi}).
\ee 
Indeed, set $F(\varphi)=\delta_{\Gamma_r,\varphi}$. By \eqref{var}, for small $t$,
$$
    F(\varphi+t\eta)
    \ge
    \frac{\psi_{\Gamma_r}(v_{r,\varphi})}{(\varphi+t\eta)(v)}
    =
    \frac{F(\varphi)}{1+t\eta(v_{r,\varphi})}.
$$
Dividing by $t>0$ and letting $t\to0^+$ gives
$D_\varphi F(\eta)\ge -F(\varphi)\eta(v_{r,\varphi}).$
Applying the same argument with $-\eta$ in place of $\eta$ gives the reverse inequality, yielding \eqref{dvv}.

The same formula holds for the model group:
$$
    D_\varphi\delta_{\Gamma_\infty,\varphi}(\eta)
    =
    -\delta_{\Gamma_\infty,\varphi}\,\eta(v_{\infty,\varphi}).
$$
Since $\delta_{\Gamma_r,\varphi}\to\delta_{\Gamma_\infty,\varphi}>0$
uniformly on $\mathcal U$, the convergence of differentials implies that for every $\eta\in\fa^*$,
$$
    \eta(v_{r,\varphi})\to \eta(v_{\infty,\varphi})
    \qquad\text{uniformly for }\varphi\in\mathcal U.
$$
 Choosing a basis of $\fa^*$, we conclude that
$ v_{r,\varphi}\to v_{\infty,\varphi}$ uniformly for $\varphi\in\mathcal U$.

\end{proof}

\begin{lemma}\label{lem:maximizer-near-barycenter}
Suppose that
$\lambda({\mathsf c})\ne \op{i}(\lambda({\mathsf c}))$. Then there exist a
conic neighborhood $U$ of the barycentric ray in $\fa^+$, a neighborhood
$\mathcal U$ of $\varphi_0$ in $\fa^*$, and $r_0\ge1$ such that, for
all $r\ge r_0$ and all $\varphi\in\mathcal U$, 
 $$v_{r,\varphi}\in U\quad\text{and}\quad 
    U\subset \operatorname{int}\L_{N'_r}
    \quad\text{for all } r\ge r_0.
$$
\end{lemma}

\begin{proof}
Let $U$ be a conic neighborhood of
the barycentric ray whose closure is contained in the interior of the convex cone spanned by 
 $\la({\mathsf c})$ and $\la({\mathsf c}^{-1})$.
Since $c,c^{-1}\in\mathcal N'$,  the limit cone $\L_{N_r'}$ contains the convex cone spanned  by
 $\la({\mathsf c})$ and $\la({\mathsf c}^{-1})$.
Therefore
$U\subset \operatorname{int}\L_{N_r'}.$

It remains to show that $v_{r,\varphi}\in U$ for $\varphi$ near
$\varphi_0$ and $r$ large. Since $\varphi_0$ and $\psi_{\Ga_\infty}$ are both $\i$-invariant, it follows from the uniqueness that $v_{\infty,\varphi_0}$ lies on the barycentric ray.
 The map
$$
    \varphi\mapsto v_{\infty,\varphi}
$$
is continuous near $\varphi_0$. Indeed, since $\varphi\mapsto\delta_{\Gamma_\infty,\varphi}$ is real analytic near
$\varphi_0$, and since $v_{\infty, \varphi}$ satisfies
$$
    D_\varphi\delta_{\Gamma_\infty,\varphi}(\eta)
    =
    -\delta_{\Gamma_\infty,\varphi}\,\eta(v_{\infty,\varphi})
    \qquad(\eta\in\fa^*),
$$
the vector $v_{\infty, \varphi}$ depends continuously on $\varphi$.
Therefore, after shrinking
$\mathcal U$, we have
$$
    v_{\infty,\varphi}\in U
    \qquad(\varphi\in\mathcal U).
$$
By the uniform convergence $v_{r,\varphi}\to v_{\infty,\varphi}$ (\Cref{lem:uniform-maximizer-convergence}), after
increasing $r_0$ we obtain
$v_{r,\varphi}\in U$
 for all $r\ge r_0$ and $\varphi\in\mathcal U$.
This proves the lemma.
\end{proof}

\GrowthIndicatorGap*

The rest of this section is devoted to its proof.
Choose $\mathsf c$ as in Lemma \ref{lem:maximizer-near-barycenter}. Thus there
exist a neighborhood $U$ of the barycentric ray, a neighborhood
$\mathcal U$ of $\varphi_0$, and $r_0\ge1$ such that, for all $r\ge r_0$,
$$
    U\subset \operatorname{int}\L_{N'_r}.
$$

Choose a nonsymmetric form $
    \varphi=s_1\omega_1+s_2\omega_2\in \mathcal U$ with
    $s_1,s_2>0$ and $s_1\ne s_2$.
By the proof of \Cref{thm:omega1-gap},  after increasing $r$ if necessary,
we have
$$
    \delta_{N'_r,\varphi}
    <
    \delta_{\Ga_r,\varphi}.
$$
Fix such an $r$, and set 
$$
    \Gamma:=\Ga_r,
    \quad\text{and}\quad
    N:=N'_r.
$$
For $r$ sufficiently large, $\Gamma$ is Zariski dense and Borel-Anosov,
and $N\lhd\Gamma$ is cocyclic and Zariski dense.
Let $v=v_\varphi$ be the unique  maximizing vector for
$\delta_{\Gamma,\varphi}$, normalized by $\varphi(v)=1$. By
Lemma \ref{lem:maximizer-near-barycenter}, we have
$$
    v\in U\subset \operatorname{int}\L_N.
$$

We claim that
\be\label{dif}
    \psi_N(v)\ne \psi_\Gamma(v).\ee 
Indeed, if $\psi_N(v)=\psi_\Gamma(v)$, then since $v\in \L_N$ and
$\varphi(v)=1$, by Lemma \ref{KMO_tent} we would have
$$
    \delta_{N,\varphi}
    \ge
    \psi_N(v)
    =
    \psi_\Gamma(v)
    =
    \delta_{\Gamma,\varphi},
$$
contradicting $\delta_{N,\varphi}<\delta_{\Gamma,\varphi}$. This proves \eqref{dif}.

We now show that one may further arrange that
$$
    \mathcal L_N=\mathcal L_\Gamma .
$$
For this, we use the following extension of
\cite[Theorem A.1]{DeyHurtado25}, where the subgroup $\Delta$ was
assumed to be cyclic. The same argument applies to a finitely generated
free factor once the required ping-pong configuration is arranged. 
\begin{lemma}\label{lem:same-cone-finite-index}
Let $\Gamma<\SL_3(\mathbb R)$ be a Zariski dense Borel-Anosov subgroup. Suppose that
$\Gamma$ splits as a free product
$$
    \Gamma=\Delta*\langle \mathsf c\rangle,
$$
where $\Delta<\Gamma$ is a finitely generated subgroup and
$\mathsf c\in\Gamma$ has infinite order with $\la(\mathsf c)\ne \la (\mathsf c^{-1})$.
 Suppose that
    \begin{enumerate}
        \item $\lambda(\mathsf c)$ lies outside the smallest closed convex cone in $\fa^+$ containing $\mathcal L_{\Delta}$;
        \item for all finite index subgroups $\Delta'$ of $\Delta$, the subgroup $\langle \Delta', \mathsf c\rangle$ is Zariski dense.
    \end{enumerate} 
    Then there exists a finite index subgroup $\Delta'$ of $\Delta$ such that the limit cone $\mathcal L_{\Delta'* \langle\mathsf c\rangle}$ is equal to the convex cone bounded by the two rays containing $\la(\mathsf c)$ and $\la (\mathsf c^{-1})$.
 
\end{lemma}

\begin{proof} Let $\mathcal C_{\mathsf c}$ denote the convex cone spanned by two rays
$\la(\cc)$ and $\la(\cc^{-1})$. Since $\L_\Delta$ is $\i$-invariant, so is the smallest convex cone containing it. Since $\i (\la (\cc))=\la(\cc^{-1})$,
 the first hypothesis implies that $\la(\cc^{-1})$ lies outside the smallest convex cone containing $\L_\Delta$. Therefore we have  $\L_\Delta-\{0\}\subset \inte \mathcal C_{\mathsf c}$. 

We first arrange a ping-pong configuration.  Since $\Gamma$ is
Borel-Anosov, its boundary map to the full flag variety $\mathcal F$ is antipodal.
 Since
$\Delta$ is a free factor of the word-hyperbolic group $\Gamma$, it is
quasiconvex in $\Gamma$. Hence $\Delta$ is again Borel-Anosov, with
limit set $\Lambda_\Delta\subset \cal F$. 

Let $\xi_{\cc}\in \mathcal F$ and $\xi_{\cc^{-1}}\in \mathcal F$ denote the attracting fixed points of $\cc$ and $\cc^{-1}$ respectively.
 Choose compact neighborhoods $Y$ of $\Lambda_\Delta$ and
$X$ of
$    \{\xi_{\cc} ,\xi_{\cc^{-1}}\}\cup
    \bigcup_{n\ne0}\mathsf c^n\Lambda_\Delta$
so that every flag in $X$ is antipodal to every flag in $Y$, and so that for all $n\in \Z-\{0\}$,
$$
    \mathsf c^nY\subset\operatorname{int}X
$$
By the convergence dynamics of the Anosov subgroup $\Delta$, all but
finitely many elements $g\in\Delta$ satisfy
$$
    gX\subset\operatorname{int}Y.
$$
Since $\Delta$ is finitely generated and linear, it is residually finite.
Passing to a finite-index subgroup $\Delta'<\Delta$ which avoids this
finite exceptional set, we obtain
$$
    gX\subset\operatorname{int}Y
    \qquad \text{for all $g\in\Delta'-\{e\}$.}
$$
Thus $\Delta'$ and $\langle\mathsf c\rangle$ play ping-pong with the
sets $X$ and $Y$.

By \cite[Corollary 3.8]{DeyHurtado25} (see also
\cite[Section 4.1]{Benoist97}), there exists a constant $C_0>0$ such
that every reduced word
$$
    w=g_1\mathsf c^{n_1}\cdots g_k\mathsf c^{n_k},
    \qquad
    g_i\in\Delta'-\{e\},\quad n_i\in\mathbb Z-\{0\},
$$
satisfies
$$
    \left\|
    \mu(w)
    -
    \sum_{i=1}^k\mu(g_i)
    -
    \sum_{i=1}^k\mu(\mathsf c^{n_i})
    \right\|
    \le C_0k.
$$
Since
$$
    \|\mu(\mathsf c^n)-\lambda(\mathsf c^n)\|=O(1)
    \qquad(n\in\mathbb Z),
$$
we may enlarge $C_0$ and write
$$
    \left\|
    \mu(w)
    -
    \sum_{i=1}^k\mu(g_i)
    -
    \sum_{i=1}^k |n_i|\lambda(\mathsf c^{\operatorname{sgn}n_i})
    \right\|
    \le C_0k.
$$

If $w$ is cyclically reduced, applying this estimate to $w^m$, dividing
by $m$, and letting $m\to\infty$, gives
\be\label{give}
    \left\|
    \lambda(w)
    -
    \sum_{i=1}^k\mu(g_i)
    -
    \sum_{i=1}^k |n_i|\lambda(\mathsf c^{\operatorname{sgn}n_i})
    \right\|
    \le C_0k.
\ee

Indeed, when $w$ is cyclically reduced, the word $w^m$ is obtained by
concatenating $m$ copies of $w$ without cancellation, and $\lambda(w)=\lim_{m\to\infty}\frac1m\mu(w^m)$.

If $w$ is not cyclically reduced, then $w$ is conjugate either to a
cyclically reduced word or to an element of one of the free factors. Since
the Jordan projection is invariant under conjugation, replacing $w$ by a
conjugate does not change $\lambda(w)$. Thus the preceding estimate
applies after conjugating $w$ to cyclically reduced form; if $w$ is
conjugate into $\Delta'$ or into $\langle \mathsf c\rangle$, the desired
cone containment follows directly from $ \mathcal L_{\Delta'}\subset \mathcal C_{\mathsf c}$ or $
    \lambda(\mathsf c^{\pm1})\in\mathcal C_{\mathsf c}$.

We now show that every Jordan projection $\la(w)$  lies in $\mathcal C_{\mathsf c}$.
For $v\in\operatorname{int}\mathfrak a^+$, set
$$
    R(v):=\frac{\alpha_1(v)}{\alpha_2(v)}.
$$
Let
$$
    R_+:=
    \max\{R(\lambda(\mathsf c)),R(\lambda(\mathsf c^{-1}))\},
    \qquad
    R_-:=
    \min\{R(\lambda(\mathsf c)),R(\lambda(\mathsf c^{-1}))\}.
$$
Then $\mathcal C_{\mathsf c}$ is precisely the closed cone of vectors
$v\in\mathfrak a^+$ satisfying
$$
    R_-\le R(v)\le R_+.
$$

Since $$\mathcal L_\Delta-\{0\}
    \subset
    \operatorname{int}\mathcal C_{\mathsf c},$$
the ratios $R(v)$, for $v\in\mathcal L_\Delta-\{0\}$, are uniformly
bounded away from $R_+$ and $R_-$. Moreover, since $\mathcal L_\Delta$  is the asymptotic cone of $\mu(\Delta)$, after passing to a deeper
finite-index subgroup $\Delta'$, if necessary, we may therefore assume
that, for every non-trivial $g\in\Delta'$,
\be\label{muc} 
    \alpha_2(\mu(g))>C_0
\quad\text{ and }\quad 
    \frac{\alpha_1(\mu(g))+C_0}
         {\alpha_2(\mu(g))-C_0}
    <
    R_+.
\ee 

So
$$
D_w:=
\sum_{i=1}^k\bigl(\alpha_2(\mu(g_i))-C_0\bigr)
+
\sum_{i=1}^k |n_i|
\alpha_2\bigl(\lambda(\mathsf c^{\operatorname{sgn} n_i})\bigr) >0.
$$
 Since
 $ {\alpha_1(\lambda(\mathsf c^{\pm1}))}\le R_+
         {\alpha_2(\lambda(\mathsf c^{\pm1}))}$ and
 $ {\alpha_1(\mu(g))+C_0}< R_+
         {\alpha_2(\mu(g))-C_0}$ for all $g\in \Delta'$,
the estimate for $\lambda(w)$ in \eqref{give} gives
$$
    \alpha_1(\lambda(w))\le R_+D_w,
    \qquad
    \alpha_2(\lambda(w))\ge D_w.
$$
Hence
$$
    R(\lambda(w))\le R_+.
$$

Applying the same argument to $w^{-1}$, and using
$\lambda(w^{-1})=\op{i}(\lambda(w))$
gives the lower bound
$$
    R(\lambda(w))\ge R_-.
$$
Therefore
$$
    \lambda(w)\in\mathcal C_{\mathsf c}
    \qquad
    \text{for all $w\in\Delta'*\langle\mathsf c\rangle$}.
$$

By hypothesis, $\Delta'*\langle\mathsf c\rangle$ is Zariski dense in
$\SL_3(\mathbb R)$. Hence, its limit cone is the
smallest closed convex cone containing the Jordan projections of its
elements. Since all these Jordan projections lie in $\mathcal C_{\mathsf c}$,
we get
$$
    \mathcal L_{\Delta'*\langle\mathsf c\rangle}
    \subset
    \mathcal C_{\mathsf c}.
$$
The reverse inclusion holds because $\mathsf c$ and $\mathsf c^{-1}$
belong to $\Delta'*\langle\mathsf c\rangle$. 
This proves the lemma.
\end{proof}

In the construction above, we may additionally require that $\lambda(\mathsf c)$ lies outside $\mathcal L_{\rho_\kappa(\langle a,b\rangle)}$ and that for all $m\in\mathbb N$, $\langle \rho_{\kappa}(\langle a^m,b^m\rangle), \mathsf c \rangle$ is Zariski dense. In this case, \Cref{lem:same-cone-finite-index} implies that there exists $m\in\mathbb N$ such that $\mathcal L_{\widehat\rho_{\kappa,r}(\langle a^m,b^m,c\rangle)}$ is equal to the convex cone $\mathcal L$ bounded by the two rays containing $\lambda(\mathsf c)$ and $\lambda(\mathsf c^{-1})$.

We now repeat the preceding argument for the restriction $\widehat\rho_{\kappa,r}\vert_{\langle a^m,b^m,c\rangle}$. Note that every Zariski dense subgroup of $\widehat\rho_{\kappa,r}(\langle a^m,b^m,c\rangle)$ containing a nontrivial power of $\mathsf c$ also has the same limit cone $\mathcal L$. In particular, in the proof of the first part, we may take
$$
\Gamma = \widehat\rho_{\kappa,r}(\langle a^m,b^m,c\rangle)
\quad\text{and}\quad
N = \widehat\rho_{\kappa,r}(\langle\!\langle b^m,c\rangle\!\rangle),
$$
and obtain $\mathcal{L}_N=\mathcal{L}_\Gamma$. Here $\langle\!\langle b^m,c\rangle\!\rangle$ denotes the normal closure of $\langle b^m,c\rangle$ in $\langle a^m,b^m,c\rangle$. The remainder of the argument is unchanged. This completes the proof of Theorem \ref{thm:growth-indicator-gap}.

\begin{Rmk}
The appendix argument also gives a directional lower bound for growth
indicators. Let $N\lhd\Gamma$ be a Zariski dense normal subgroup of a
Borel-Anosov subgroup \(\Gamma<G\). Then
\be\label{normalg}
    \psi_N(v)\ge \frac{1}{2} \psi_\Gamma(v) \quad\text{for every $v\in\mathfrak a^+$ with $\op{i}(v)=v$.}
\ee

Indeed, the proof is a cone-localized version of the conjugacy-counting
argument in the appendix. Choose a non-trivial element $u\in N$, and write $u=w^p$ with $w\in\Gamma$ primitive and $p\ge1$. Since $N$ is
normal, all conjugates $fw^pf^{-1}$ lie in $N$. After choosing the
double-coset representatives $f$ as in the appendix, the map $f\mapsto fw^pf^{-1}$ has uniformly bounded multiplicity, and coarse
Cartan additivity gives
\[
    \mu(fw^pf^{-1})=\mu(f)+\mu(f^{-1})+O(1)
    =
    \mu(f)+\op{i}\mu(f)+O(1).
\]
Now suppose that \(v\in\mathfrak a^+\) satisfies \(\op{i}(v)=v\). If
\(\mu(f)\) lies in a sufficiently small cone around \(v\), then
\(\mu(f)+\op{i}\mu(f)\) lies in a prescribed cone around the same ray
\(\mathbb R_{>0}v\), and its norm is \(2\|\mu(f)\|+O(1)\). Hence, for every
open cone \(\mathcal C\) containing \(v\), there is a smaller open cone
\(\mathcal C'\) containing \(v\) such that, up to a polynomial factor in
\(T\),
\[
\begin{aligned}
&\#\left\{h\in N:
    \begin{array}{c}
    \mu(h)\in\mathcal C,\\
    \|\mu(h)\|\le 2T+O(1)
    \end{array}
    \right\} \\
&\hspace{35mm}\gg
    \#\left\{\gamma\in\Gamma:
    \begin{array}{c}
    \mu(\gamma)\in\mathcal C',\\
    \|\mu(\gamma)\|\le T
    \end{array}
    \right\}.
\end{aligned}
\]
Taking logarithms, dividing by \(2T\), and then letting \(T\to\infty\), gives
\[
    \tau_{\mathcal C}(N)\ge \frac12\,\tau_{\mathcal C'}(\Gamma).
\]
Here
\[
\tau_{\cal C}(\star)
=
\limsup_T\frac1T\log
\#\{g\in\star:\mu(g)\in\cal C,\ \|\mu(g)\|<T\}.
\]
Finally, taking the infimum over cones \(\mathcal C\) containing \(v\) yields
$\psi_N(v)\ge \frac12\psi_\Gamma(v)$, as claimed.
\end{Rmk} 
\newpage
\appendix
\section{Critical exponents of normal subgroups}
\begin{center}
    \textsc{Konstantinos Tsouvalas}
\end{center}

\subsection*{Critical exponents of normal subgroups} 
For an arbitrary normal subgroup, without a coamenability assumption,
we have the following lower bound.
\begin{theorem}\label{m2}
Let $G$ be a connected semisimple real algebraic group, let $\Ga<G$ be a
Borel-Anosov subgroup that is not virtually cyclic, and let $N\lhd\Ga$ be
an infinite normal subgroup. Then
\[
    \delta_N\ge\frac12\delta_\Ga.
\]
\end{theorem}

This subsection is devoted to the proof. Since the critical exponent is
unchanged upon passing to a finite-index subgroup, we assume for the rest
of the appendix that $\Gamma$ is torsion-free. We fix a word metric $|\cdot|$ on $\Gamma$ and for $\gamma_1,\gamma_2\in \Gamma$, we denote by $$(\gamma_1\cdot \gamma_2)=|\gamma_1|+|\gamma_2|-|\gamma_1^{-1}\gamma_2|$$ the Gromov product of $\gamma_1,\gamma_2$.

\begin{lemma}\label{double-coset}
Let
$w\in \Gamma$ be a primitive element. There exist $M>0$ and a set
$Q$ of double coset representatives for $\langle w\rangle$ in $\Gamma$ such that, for
any $h\in Q$ non-trivial and any $r,s\in \mathbb Z$,
$$
    \big||h^{\pm 1}w^r|-|h|-|w^r|\big|\leq M \text{ and }
    \big||w^s h^{\pm 1}w^r|-|w^s h^{\pm 1}|-|w^r|\big|\leq M.
$$
\end{lemma}

\begin{proof}
Since $\langle w\rangle$ is a quasiconvex subgroup of $\Gamma$, it acts cocompactly on
$$
    \mathcal X_w
    :=
    \big(\Gamma\cup \partial\Gamma\big)
    \smallsetminus
    \{w^+,w^-\},
$$
where $w^+$ and $w^-$ are the attracting and repelling fixed points of $w$ in
$\partial\Gamma$. Choose a compact fundamental domain
$\mathcal F_1\subset \mathcal X_w$ for this action, so that
$ \mathcal X_w=\langle w\rangle \mathcal F_1$.
In particular, the points of $\mathcal F_1$ are uniformly separated from $w^+$ and
$w^-$ in the visual metric. Let $T_1:=\Gamma\cap \mathcal F_1$.
For each $t\in T_1$, write $t^{-1}=w^{m(t)}\overline t $ for some $m(t)\in\mathbb Z$ and some $\overline t\in T_1$, 
and set $ T:=\{\overline t:t\in T_1\}$.
Then any element of $T$, as well as its inverse up to left multiplication by a power of
$w$, remains uniformly away from $\{w^+,w^-\}$ in the visual metric. Choose from $T$ one
representative for each double coset of $\langle w\rangle$ in
$\Gamma - \langle w\rangle$, and denote the resulting set by $Q$. By construction, for any $h\in Q$, the points represented by $h$ and $h^{-1}$ are
uniformly away from $w^+$ and $w^-$. Hence the Gromov products    $(h^{\mp1} \cdot w^r)$, $h\in Q,\ r\in\mathbb Z$,
are uniformly bounded.  This proves the first estimate.
The same compactness argument shows that since for any $h\in Q$ non-trivial, $h\{w^{+},w^{-}\}$ is uniformly away from $\{w^{+},w^{-}\}$ in $\Gamma \cup \partial \Gamma$, the Gromov products
$\big((w^s h^{\pm1})^{-1}\cdot w^r\big)$, $ h\in Q,\ r,s\in\mathbb Z$,
are uniformly bounded. Equivalently, $\big||w^s h^{\pm 1}w^r|-|w^s h^{\pm 1}|-|w^r|\big|$
is uniformly bounded. Enlarging the constant if necessary completes the proof.
\end{proof}

\begin{lemma}\label{exp-coset}
 For any primitive element $w\in\Gamma$, there exist
a set $Q\subset\Gamma$ of double coset representatives for $\langle w\rangle$ in
$\Gamma$ and a constant $c>0$ such that, for all $T\geq 1$,
\begin{equation}\label{exponent-ineq1}
    \#\{g\in Q:\|\mu(g)\|\leq T\}
    \geq
    {c}{T^{-2}}
    \#\{g\in\Gamma:\|\mu(g)\|\leq T\}.
\end{equation}

\end{lemma}

\begin{proof}
Let $Q\subset\Gamma$ be the set of double coset representatives provided by
Lemma \ref{double-coset}. Since $Q$ is a set of representatives for
$\langle w\rangle\backslash
    \big(\Gamma\smallsetminus \langle w\rangle\big)/
    \langle w\rangle$
any element of $\Gamma\smallsetminus\langle w\rangle$ can be written uniquely in the
form
$$
    w^r f w^s,
    \qquad r,s\in\mathbb Z,\quad f\in Q.
$$

By Lemma \ref{double-coset}, there exists $C>0$ such that, for any
$r,s\in\mathbb Z$ and any $f\in Q$,
$$
    \big||w^r f w^s|-|f|-|w^r|-|w^s|\big|\leq C.
$$
Applying Proposition \ref{Grprod}, we obtain a constant $L>0$ such that
\begin{equation}\label{exponent-ineq4}
    \sup_{r,s\in\mathbb Z}
    \sup_{f\in Q}
    \big\|
    \mu(w^r f w^s)-\mu(f)-\mu(w^r)-\mu(w^s)
    \big\|
    \leq L.
\end{equation}

Since $\langle w\rangle$ is
quasi-isometrically embedded, there exists $C_1>0$ such that, for all $T\geq 1$,
$$
    \#\{r\in\mathbb Z:\|\mu(w^r)\|\leq T\}\leq C_1T.
$$
Using \eqref{exponent-ineq4}, we obtain, after enlarging constants if necessary,
\begin{multline*}
    \#\{g\in\Gamma:\|\mu(g)\|\leq T\}
    \leq
    \#\{r\in\mathbb Z:\|\mu(w^r)\|\leq T\} \\
    \quad+
    \#\Big\{
    (r,s,f)\in\mathbb Z^2\times Q:
    \|\mu(w^r)\|+\|\mu(w^s)\|+\|\mu(f)\|\leq T+L
    \Big\}  \\
    \leq
    C_2T^2
    \#\{f\in Q:\|\mu(f)\|\leq T+L\}+C_2T .
    \end{multline*}

Absorbing the fixed additive error $L$ and the lower-order term into the polynomial
factor proves \eqref{exponent-ineq1}. 
\end{proof}

\begin{proof}[Proof of Theorem \ref{m2}]
Since $\Gamma$ is torsion-free and hyperbolic and $N$ is not virtually
cyclic, there exists a primitive element $w\in N$; see, for example, the
proof of \cite[Theorem 7.7]{IKapovich}.

Let $Q$ be the set of double coset representatives of
$\langle w\rangle$ given by Lemma \ref{exp-coset}.
Since $\Gamma$ is torsion-free hyperbolic and
$w$ is primitive, the centralizer of $w$ is $\langle w\rangle$.  Hence
the map $Q\to \Gamma$ given by $ f\mapsto fw f^{-1} $
is injective.
 Moreover, by subadditivity of the Cartan projection and by
$\|\mu(f^{-1})\|=\|\mu(f)\|$, we have
$$
    \|\mu(fw f^{-1})\|
    \leq
    2\|\mu(f)\|+\|\mu(w)\|.
$$
Hence, for any $T\geq 1$,
$$
    \#\{h\in N:\|\mu(h)\|\leq 2T+\|\mu(w)\|\}
    \geq
    \#\{f\in Q:\|\mu(f)\|\leq T\}.
$$
Using \eqref{exponent-ineq1}, we obtain
$$
    \#\{h\in N:\|\mu(h)\|\leq 2T+\|\mu(w)\|\}
    \geq
 {c}{T^{-2}}
    \#\{g\in\Gamma:\|\mu(g)\|\leq T\}.
$$
Taking logarithms, dividing by $2T+\|\mu(w)\|$, and passing to the upper limit gives $\delta_N\geq \delta_\Gamma/2 $.
\end{proof}

We remark that  estimates similar to Theorem \ref{m2} have been established for the growth of confined subgroups in discrete subgroups of isometry groups of Gromov hyperbolic and $\textup{CAT}(0)$ spaces, see \cite[Theorem 1.3]{CGYZ24}. We would like to thank Inhyeok Choi for pointing out to us the paper \cite{CGYZ24}.


\begin{thebibliography}{10}

\bibitem{AbelsMargulisSoifer}
Herbert Abels, Gregory~A. Margulis, and Gregory~A. So\u{i}fer.
\newblock Semigroups containing proximal linear maps.
\newblock {\em Israel J. Math.}, 91(1-3):1--30, 1995.

\bibitem{burger_Inv}
Norbert A'Campo and Marc Burger.
\newblock R\'{e}seaux arithm\'{e}tiques et commensurateur d'apr\`es {G}. {A}.
  {M}argulis.
\newblock {\em Invent. Math.}, 116(1-3):1--25, 1994.

\bibitem{Benoist96}
Yves Benoist.
\newblock Actions propres sur les espaces homog\`enes r\'eductifs.
\newblock {\em Ann. of Math. (2)}, 144(2):315--347, 1996.

\bibitem{Benoist97}
Yves Benoist.
\newblock Propri\'et\'es asymptotiques des groupes lin\'eaires.
\newblock {\em Geom. Funct. Anal.}, 7(1):1--47, 1997.

\bibitem{CGYZ24}
Inhyeok Choi, Ilya Gekhtman, Wenyuan Yang, and Tianyi Zheng.
\newblock Confined subgroups in groups with contracting elements.
\newblock Preprint, {arXiv}:2405.09070, 2024.

\bibitem{ShapiraEtAll25}
R\'emi Coulon, Rhiannon Dougall, Barbara Schapira, and Samuel Tapie.
\newblock Twisted {P}atterson-{S}ullivan measures and applications to
  amenability and coverings.
\newblock {\em Mem. Amer. Math. Soc.}, 305(1539):v+93, 2025.

\bibitem{DeyHurtado25}
Subhadip Dey and Sebastian Hurtado.
\newblock Remarks on discrete subgroups with full limit sets in higher rank
  {L}ie groups.
\newblock {\em Int. Math. Res. Not. IMRN}, (17):Paper No. rnaf268, 26, 2025.

\bibitem{DKL19}
Subhadip Dey, Michael Kapovich, and Bernhard Leeb.
\newblock A combination theorem for {A}nosov subgroups.
\newblock {\em Math. Z.}, 293(1-2):551--578, 2019.

\bibitem{DKO_AR}
Subhadip Dey, Dongryul~M. Kim, and Hee Oh.
\newblock Ahlfors regularity of {P}atterson-{S}ullivan measures of {A}nosov
  groups and applications.
\newblock {\em Compos. Math.} Vol 162 (2026), 743-791.


\bibitem{DeyOh25}
Subhadip Dey and Hee Oh.
\newblock Deformations of {A}nosov subgroups: limit cones and growth
  indicators.
\newblock {\em J. Lond. Math. Soc. (2)}, 112(3):Paper No. e70280, 35, 2025.


\bibitem{ELO_anosov}
Samuel Edwards, Minju Lee, and Hee Oh.
\newblock Anosov groups: local mixing, counting and equidistribution.
\newblock {\em Geom. Topol.}, 27(2):513--573, 2023.

\bibitem{FockGoncharov11}
Vladimir Fock and Alexander Goncharov.
\newblock Moduli spaces of local systems and higher {T}eichm\"uller theory.
\newblock {\em Publ. Math. Inst. Hautes \'Etudes Sci.}, (103):1--211, 2006.

\bibitem{GTa}
Olivier Glorieux and Samuel Tapie.
\newblock Critical exponents of normal subgroups in higher rank.
\newblock {\em (arXiv:2006.05730)}.

\bibitem{GuichardWienhard12}
Olivier Guichard and Anna Wienhard.
\newblock Anosov representations: domains of discontinuity and applications.
\newblock {\em Invent. Math.}, 190(2):357--438, 2012.

\bibitem{IKapovich}
Ilya Kapovich.
\newblock A non-quasiconvexity embedding theorem for hyperbolic groups.
\newblock {\em Math. Proc. Camb. Philos. Soc.}, 127(3):461--486, 1999.

\bibitem{KLP17}
Michael Kapovich, Bernhard Leeb, and Joan Porti.
\newblock Anosov subgroups: dynamical and geometric characterizations.
\newblock {\em Eur. J. Math.}, 3(4):808--898, 2017.

\bibitem{KLP25}
Michael Kapovich, Bernhard Leeb, and Joan Porti.
\newblock Morse actions of discrete groups on symmetric spaces: local-to-global
  principle.
\newblock {\em Geom. Topol.}, 29(5):2343--2390, 2025.

\bibitem{kassel}
Fanny Kassel.
\newblock Deformation of proper actions on reductive homogeneous spaces.
\newblock {\em Math. Ann.}, 353(2):599--632, 2012.

\bibitem{kassel-potrie}
Fanny Kassel and Rafael Potrie.
\newblock Eigenvalue gaps for hyperbolic groups and semigroups.
\newblock {\em J. Mod. Dyn.}, 18:161--208, 2022.

\bibitem{KMO_tent}
D.~M. Kim, Y.~Minsky, and H.~Oh.
\newblock Tent property of the growth indicator functions and applications.
\newblock {\em Geom. Dedicata}, 218(1):Paper No.14, 18, 2024.

\bibitem{KOW_indicators}
Dongryul~M. Kim, Hee Oh, and Yahui Wang.
\newblock Properly discontinuous actions, growth indicators, and conformal
  measures for transverse subgroups.
\newblock {\em Math. Ann.}, 393(2):2391--2450, 2025.

\bibitem{Labourie06}
Fran\c{c}ois Labourie.
\newblock Anosov flows, surface groups and curves in projective space.
\newblock {\em Invent. Math.}, 165(1):51--114, 2006.

\bibitem{lahn2025}
Max Lahn.
\newblock Reducible suspensions of {A}nosov representations.
\newblock {\em (arXiv:2312.09886), Groups Geom. Dyn.}
\newblock To appear.

\bibitem{LeeOh23}
Minju Lee and Hee Oh.
\newblock Invariant measures for horospherical actions and {A}nosov groups.
\newblock {\em Int. Math. Res. Not. IMRN}, (19):16226--16295, 2023.



\bibitem{Quint02}
Jean-Fran\c{c}ois Quint.
\newblock Divergence exponentielle des sous-groupes discrets en rang
  sup\'erieur.
\newblock {\em Comment. Math. Helv.}, 77(3):563--608, 2002.

\bibitem{Quint_indicator}
Jean-Fran\c{c}ois Quint.
\newblock L'indicateur de croissance des groupes de {S}chottky.
\newblock {\em Ergodic Theory Dynam. Systems}, 23(1):249--272, 2003.

\bibitem{Quint_Schottky}
Jean-Fran\c{c}ois Quint.
\newblock Groupes de {S}chottky et comptage.
\newblock {\em Ann. Inst. Fourier (Grenoble)}, 55(2):373--429, 2005.

\bibitem{Roblin05}
Thomas Roblin.
\newblock Un th\'eor\`eme de {F}atou pour les densit\'es conformes avec
  applications aux rev\^etements galoisiens en courbure n\'egative.
\newblock {\em Israel J. Math.}, 147:333--357, 2005.

\bibitem{Sambarino14}
Andr\'{e}s Sambarino.
\newblock Hyperconvex representations and exponential growth.
\newblock {\em Ergodic Theory Dynam. Systems}, 34(3):986--1010, 2014.

\bibitem{sambarino2022report}
Andr\'{e}s Sambarino.
\newblock A report on an ergodic dichotomy.
\newblock {\em Ergodic Theory Dynam. Systems}, 44(1):236--289, 2024.


\bibitem{Tits_classification}
Jacques Tits.
\newblock Classification of algebraic semisimple groups.
\newblock {\em Proc. Sympos. Pure. Math.}, 9:33--62, 1966.

\bibitem{Tits_free}
Jacques Tits.
\newblock Free subgroups in linear groups.
\newblock {\em J. Algebra}, 20:250--270, 1972.



\end{thebibliography}
\end{document}